\documentclass[english]{smfart}
\usepackage{amssymb}
\usepackage{euscript}
\usepackage[hypertex]{hyperref}

\title[Operads of natural operations I]%
       {Operads of natural operations I: 
\\Lattice paths, braces and Hochschild cochains}

\alttitle{Op\'erades des op\'erations naturelles I:  chemins bris\'es, 
op\'erations brace et cocha\^\i nes de Hochschild}

\author{Michael Batanin}
\address{Macquarie University, North Ryde, 2109, NSW, Australia}
\email{mbatanin@ics.mq.edu.au}
\author{Clemens Berger}
\address{Universit\'e de Nice, Lab. J.-A. Dieudonn\'e, Parc Valrose, 
F-06108 Nice, France}
\email{cberger@math.unice.fr}
\author{Martin~Markl}
\address{Mathematical Institute of the Academy, {\v Z}itn{\'a} 25,
         115 67 Prague 1, The Czech Republic}
\email{markl@math.cas.cz}
\thanks{M.~Batanin was supported by Scott Russell Johnson Memorial Fund and Australian Research Council Grant DP0558372.}
\thanks{C.~Berger was supported by the grant OBTH of the French Agence Nationale de Recherche}
\thanks{M.~Markl was supported by the grant GA \v CR
201/08/0397 and by the Academy of Sciences of the Czech Republic,
Institutional Research Plan No.~AV0Z10190503.}

\swapnumbers
\newtheorem{theorem}{Theorem}[section]

\newtheorem{lemma}[theorem]{Lemma}
\newtheorem{proposition}[theorem]{Proposition}

\theoremstyle{definition}
\newtheorem{example}[theorem]{Example}
\newtheorem{remark}[theorem]{Remark}
\newtheorem{definition}[theorem]{Definition}
\newtheorem{convention}[theorem]{Convention}
\newtheorem{cvo}[theorem]{Classical vs.~operadic}

\newtheorem{variant}[theorem]{Variant}
\newtheorem{unnorm}[theorem]{Un-normalized totalizations}
\newtheorem{norm}[theorem]{Normalized totalizations}
\newtheorem{semi}[theorem]{Semi-normalizations}
\newtheorem{latpath}[theorem]{Marked lattice paths}
\newtheorem{simpl}[theorem]{Simplicial structures}
\newtheorem{cosimpl}[theorem]{The cosimplicial structure}

\newtheorem{TT}[theorem]{Tamarkin-Tsygan operad}

\newtheorem*{notation}{Notation}

\newtheorem{the-big}[theorem]{The operad of natural operations}

\def\sqzld{{\mbox{$. \hskip -1pt . \hskip -1pt .$}}}
\def\NormB{\Norm(\Big)}\def\HB{\widehat{\Big}}
\def\NormT{\Norm(\Tam)}\def\nuBr{\widehat{\Br}}
\def\Mon{{\mathcal M}}\def\Bigclas{{\mathcal M}_{\it cl}}\def\In{{\it In}}
\def\Hom{{\it Hom\/}}\def\Nat{{\it Nat}}\def\bbT{{\mathbb T}}
\def\Big{{\mathcal B}}\def\Bigclas{{\mathcal B}_{\it cl}}\def\In{{\it In}}
\def\dh{d_H}\def\Dis{{\EuScript D}}\def\whis{{\it w\/}}
\def\Tam{{\mathcal T}}\def\uN{\underline{\rm Nor\/}} \def\nA{{\widehat A}}
\def\nBr{{\widehat \Big}}\def\nTam{{\widehat \Tam}}\def\Bn{{\widehat B}}
\def\NBrr#1{\widehat{\Br}^{\raisebox{-.35em}{\tiny $#1$}}}
\def\NBr{\Norm(\Br)}\def\HBr{\widehat{\Br}} \def\Chain{\EuScript Chain}
\def\Br{{\EuScript Br}}\def\Tn{{\widehat T}}
\def\UAss{\hbox{$U \hskip -.2em  \it{\mathcal A}ss$}}
\def\CH#1{C^{#1}(A;A)}\def\CHclas#1{C_{\it cl}^{#1}(A;A)}
 \def\br#1{{\it br_{#1}}}
\def\Rada#1#2#3{#1_{#2},\dots,#1_{#3}}\def\rada#1#2{{#1,\ldots,#2}}
\def\ot{\otimes} \def\pa{\partial}\def\SC{{\mathit C}}
\def\exeptional{$\ \unitlength 1em \thicklines\line(0,1){1}\ $}
\def\id{{\mathrm {{id}}}}\def\otexp#1#2{{#1}^{\otimes {#2}}}
\def\Span{{\rm Span}}\def\amp{{\it amp\/}}\def\oN{\overline{\rm Nor\/}}
\def\Lin{{\it Lin\/}}\def\proj{{\it proj\/}}
\def\Endop{{\mathcal E}{\hskip -.2em \it nd\/}}\def\Angl{{\it Angl}}
\def\desusp{\downarrow \hskip -.2em}\def\ss{{\mathbf s\, }}
\def\da{\downarrow}\def\sgn{{\rm sgn}}\def\Im{{\it Im}}
\def\lstretch#1{\left(\rule{0pt}{#1em}\right.}
\def\rstretch#1{\left.\rule{0pt}{#1em}\right)}
\def\cupop{\hbox{$\raisebox{.05em}{\mbox{-}}\hskip-.2em 
            \cup\hskip-.2em\raisebox{.05em}{\mbox{-}}$}}
\def\circop{\hbox{$\raisebox{.05em}{\mbox{-}}\hskip-.2em 
            \circ\hskip-.2em\raisebox{.05em}{\mbox{-}}$}}
\def\stub{\ 
\unitlength 1em 
\begin{picture}(0,1)
\thicklines
\put(0,0){\line(0,1){1}}
\put(0,0){\makebox(0,0){$\bullet$}}
\end{picture}
}
\def\doubless#1#2{{
\def\arraystretch{.5}
\begin{array}{c}
\mbox{\scriptsize $\scriptstyle #1$}
\\
\mbox{\scriptsize $\scriptstyle #2$}
\end{array}\def\arraystretch{1}
}}
\def\Lat#1{{\EuScript L_{(#1)}}} \def\Latnic{{\EuScript L}} 
 
\def\catstar{{\it Cat}_{*,*}} \def\op{{\rm op}}
\def\INT{{\it Int}}\def\cub{{\EuScript Q}}
\def\uTot{\underline{\rm Tot}}\def\EE{{\overline E}}\def\dd{{\overline d}}
\def\oTot{\overline{\rm Tot}}
\def\nLat#1{\overline{\EuScript L}_{(#1)}} \def\Norm{{\rm Nor\/}} 
\def\sLat#1{\dot{\EuScript L}_{(#1)}}\def\BBar{{\sf B}}
\def\Brac#1{{\EuScript Br}_{(#1)}}
\def\nBrac#1{\Norm({\EuScript Br}_{(#1)}\hskip -.05em)}
\def\HBrac#1{\widehat{\EuScript Br}_{(#1)}}
\def\A{{{\mathbb Z}[x]}}\def\nonneg{{\mathbb N}}
\def\nat{{\mathbb N}}\def\abel{{\rm Abel}}
 
\def\epi{{\twoheadrightarrow}}\def\Deg{{\it Dgn}}
\def\Coend{{\it Coend}}
\def\bbN{{\mathbb N}}
\def\cases#1#2#3#4{
                  \left\{
                         \begin{array}{ll}
                           #1,\ &\mbox{#2}
                           \\
                           #3,\ &\mbox{#4}
                          \end{array}
                   \right.
}

\def\AAd{\thicklines
{
\unitlength=.600000pt
\begin{picture}(40.00,60.00)(0.00,-10.00)
\put(20.00,5.00){\makebox(0.00,0.00)[t]
    {$\underbrace{\rule{3em}{0pt}}_{\mbox{\scriptsize $d$-times}}$}}
\put(27.00,10.00){\makebox(0.00,0.00){$\cdots$}}
\put(40.00,10.00){\makebox(0.00,0.00){$\bullet$}}
\put(10.00,10.00){\makebox(0.00,0.00){$\bullet$}}
\put(0.00,10.00){\makebox(0.00,0.00){$\bullet$}}
\put(20.00,30.00){\makebox(0.00,0.00){$\circ$}}
\put(20.00,50.00){\line(0,-1){17.00}}
\put(40.00,10.00){\line(-1,1){18.00}}
\put(10.00,10.00){\line(1,2){9.00}}
\put(0.00,10.00){\line(1,1){18.00}}
\end{picture}}
}

\def\Abracemod#1{\thicklines
{
\unitlength=.600000pt
\begin{picture}(60.00,30.00)(-10.00,15.00)
\put(27.00,6.00){\makebox(0.00,0.00){$\cdots$}}
\put(42.00,6.00){\makebox(0.00,0.00){$\circ$}}
\put(8.00,6.00){\makebox(0.00,0.00){$\circ$}}
\put(-2.00,6.00){\makebox(0.00,0.00){$\circ$}}
\put(20.00,31.00){\makebox(0.00,0.00){$\circ$}}
\put(20.00,50.00){\line(0,-1){15.00}}
\put(40.00,10.00){\line(-1,1){18.00}}
\put(10.00,10.00){\line(1,2){9.00}}
\put(0.00,10.00){\line(1,1){18.00}}
\put(25.00,40.00){\makebox(0.00,0.00)[l]{\scriptsize $1$}}
\put(-2.00,1.00){\makebox(0.00,0.00)[t]{\scriptsize $2$}}
\put(9.00,1.00){\makebox(0.00,0.00)[t]{\scriptsize $3$}}
\put(42.00,1.00){\makebox(0.00,0.00)[t]{\scriptsize $#1$}}
\end{picture}}
}

\def\Jarunka#1{
\unitlength=1.700000pt
\begin{picture}(40.00,30.00)(0.00,10.00)
\thicklines
\put(27.00,10.00){\makebox(0.00,0.00){$\cdots$}}
\put(20.00,30.00){\makebox(0.00,0.00){$#1$}}
\put(20.00,41.00){\line(0,-1){10.00}}
\put(21.00,29.50){\line(3,-2){22.00}}
\put(19.50,29.00){\line(-1,-2){7.00}}
\put(19.00,29.50){\line(-3,-2){22.0}}
\put(-10.00,9.50){\trianglepara 1}
\put(5.50,9.50){\trianglepara 2}
\put(36.00,9.50){\trianglepara d}
\end{picture}
}

\def\Jarunkamod{
\unitlength=1.700000pt
\begin{picture}(40.00,30.00)(0.00,10.00)
\thicklines
\put(27.00,10.00){\makebox(0.00,0.00){$\cdots$}}
\put(20.00,30.00){\makebox(0.00,0.00){$\bullet$}}
\put(20.00,41.00){\line(0,-1){10.00}}
\put(21.00,29.50){\line(3,-2){22.00}}
\put(5.5,20.00){\line(3,-2){7.00}}
\put(5.5,20.00){\makebox(0.00,0.00){$\bullet$}}
\put(19.00,29.50){\line(-3,-2){22.0}}
\put(-10.00,9.50){\trianglepara 1}
\put(5.50,9.50){\trianglepara 2}
\put(36.00,9.50){\trianglepara d}
\end{picture}
}

\def\AAmodifiedd{\thicklines
{
\unitlength=.800000pt
\begin{picture}(50.00,35.00)(-10.00,12.00)
\put(20.00,8.00){\makebox(0.00,0.00)[t]
    {$\underbrace{\rule{3em}{0pt}}_{\mbox{\scriptsize $d$-times}}$}}
\put(27.00,10.00){\makebox(0.00,0.00){$\cdots$}}
\put(20.00,30.00){\makebox(0.00,0.00){$\circ$}}
\put(20.00,46.00){\line(0,-1){12.00}}
\put(40.00,10.00){\line(-1,1){18.00}}
\put(10.00,10.00){\line(1,2){9.00}}
\put(0.00,10.00){\line(1,1){18.00}}
\end{picture}}
}

\def\AAmodifieddd{\thicklines
{
\unitlength=.600000pt
\begin{picture}(60.00,35.00)(-10.00,12.00)
\put(20.00,8.00){\makebox(0.00,0.00)[t]
    {$\underbrace{\rule{2.2em}{0pt}}_{\mbox{\scriptsize $d$-times}}$}}
\put(27.00,10.00){\makebox(0.00,0.00){\scriptsize $\cdots$}}
\put(20.00,30.00){\makebox(0.00,0.00){$\circ$}}
\put(25.00,35.00){\makebox(0.00,0.00)[lb]{\scriptsize 1}}
\put(20.00,46.00){\line(0,-1){12.00}}
\put(40.00,10.00){\line(-1,1){18.00}}
\put(10.00,10.00){\line(1,2){9.00}}
\put(0.00,10.00){\line(1,1){18.00}}
\end{picture}}
}

\def\white{\thicklines
{
\unitlength=.600000pt
\begin{picture}(50.00,35.00)(-10.00,12.00)
\put(25.00,11.00){\makebox(0.00,0.00){\scriptsize $\cdots$}}
\put(20.00,30.00){\makebox(0.00,0.00){$\circ$}}
\put(40.00,10.00){\line(-1,1){18.00}}
\put(10.00,10.00){\line(1,2){9.00}}
\put(0.00,10.00){\line(1,1){18.00}}
\end{picture}}
}

\def\rwhite{\thicklines
{
\unitlength=.600000pt
\begin{picture}(80.00,35.00)(-35.00,-15)
\put(0.00,0.00){\makebox(0.00,0.00){$\bullet$}}
\put(-28.00,-10.00){\makebox(0.00,0.00){\scriptsize $i$}}
\put(-20.5,-20.50){\makebox(0.00,0.00){\white}}
\put(0.00,0.00){\line(1,-1){36.00}}
\put(0.00,0.00){\line(0,1){19.00}}
\put(0.00,0.00){\line(-1,-1){17.00}}
\end{picture}}
}

\def\lwhite{\thicklines
{
\unitlength=.600000pt
\begin{picture}(80.00,35.00)(-35.00,-15)
\put(28.00,-10.00){\makebox(0.00,0.00){\scriptsize $i$}}
\put(0.00,0.00){\makebox(0.00,0.00){$\bullet$}}
\put(18.00,-20.00){\makebox(0.00,0.00){\white}}
\put(0.00,0.00){\line(-1,-1){36.00}}
\put(0.00,0.00){\line(0,1){19.00}}
\put(0.00,0.00){\line(1,-1){17.00}}
\end{picture}}
}

\def\cwhite{\thicklines
{
\unitlength=.600000pt
\begin{picture}(100.00,35.00)(-45.00,-15)
\put(0.00,4.00){\makebox(0.00,0.00){$\circ$}}
\put(10.0,9.00){\makebox(0.00,0.00){\scriptsize $i$}}
\put(5.0,-9.00){\makebox(0.00,0.00){\scriptsize $s$}}
\put(-2.00,2.00){\line(-1,-1){38.00}}
\put(-1.00,2.00){\line(-3,-4){28.5}}
\put(2.00,2.00){\line(1,-1){38.00}}
\put(0.00,8.00){\line(0,1){15.00}}
\put(0.00,1.00){\line(0,-1){19.00}}
\put(0.00,-19.00){\makebox(0.00,0.00){$\bullet$}}
\put(0.00,-19.00){\line(1,-2){8.50}}
\put(0.00,-19.00){\line(-1,-2){8.50}}
\put(25.50,-39.00){\makebox(0.00,0.00)[b]{\scriptsize $\cdots$}}
\put(-17.50,-39.00){\makebox(0.00,0.00)[b]{\scriptsize $\cdots$}}
\end{picture}}
}

\def\triangle{\thicklines
{
\unitlength=.600000pt
\begin{picture}(45.00,25.00)(-20.00,-15)
\put(0.00,0.00){\line(0,1){10.00}}
\put(0.00,0.00){\line(1,-1){18.00}}
\put(0.00,0.00){\line(-1,-1){18.00}}
\put(-18,-17){\line(1,0){36.00}}
\put(-14,-17){\line(0,-1){9.00}}
\put(-10,-17){\line(0,-1){9.00}}
\put(14,-17){\line(0,-1){9.00}}
\put(3,-24){\makebox(0.00,0.00){\scriptsize $\cdots$}}
\end{picture}}
}

\def\trianglepara#1{\thicklines
{
\unitlength=.600000pt
\begin{picture}(45.00,25.00)(-20.00,-15)
\put(0.00,0.00){\line(1,-1){18.00}}
\put(0.00,0.00){\line(-1,-1){18.00}}
\put(-18,-17){\line(1,0){36.00}}
\put(-14,-17){\line(0,-1){9.00}}
\put(-10,-17){\line(0,-1){9.00}}
\put(14,-17){\line(0,-1){9.00}}
\put(3,-24){\makebox(0.00,0.00){\scriptsize $\cdots$}}
\put(-11,-10){\makebox(0.00,0.00)[br]{\scriptsize $S_#1$}}
\end{picture}}
}

\def\verytriangle{\thicklines
{
\unitlength=.600000pt
\begin{picture}(45.00,25.00)(-20.00,-15)
\put(0.00,0.00){\line(0,1){10.00}}
\put(0.00,0.00){\line(1,-1){18.00}}
\put(0.00,0.00){\line(-1,-1){18.00}}
\put(-18,-18){\line(1,0){36.00}}
\put(-15,-18){\line(-1,-2){18.00}}
\put(-10,-18){\line(-1,-2){18.00}}
\put(15,-18){\line(1,-2){18.00}}
\end{picture}}
}

\def\amptriangle{\thicklines
{
\unitlength=.600000pt
\begin{picture}(45.00,25.00)(-20.00,-15)
\put(0.00,0.00){\line(1,-1){18.00}}
\put(0.00,0.00){\line(-1,-1){18.00}}
\put(-18,-17){\line(1,0){36.00}}
\put(-14,-17){\line(0,-1){9.00}}
\put(-10,-17){\line(0,-1){9.00}}
\put(14,-17){\line(0,-1){9.00}}
\put(3,-24){\makebox(0.00,0.00){\scriptsize $\cdots$}}
\end{picture}}
}

\def\lamptriangle{\thicklines
{
\unitlength=.600000pt
\begin{picture}(80.00,35.00)(-35.00,-15)
\put(0.00,0.00){\makebox(0.00,0.00){$\bullet$}}
\put(21.00,-18.00){\makebox(0.00,0.00){\amptriangle}}
\put(0.00,0.00){\line(-1,-1){36.00}}
\put(0.00,0.00){\line(0,1){19.00}}
\put(0.00,0.00){\line(1,-1){17.00}}
\end{picture}}
}

\def\ramptriangle{\thicklines
{
\unitlength=.600000pt
\begin{picture}(80.00,35.00)(-35.00,-15)
\put(0.00,0.00){\makebox(0.00,0.00){$\bullet$}}
\put(-10.00,-18.00){\makebox(0.00,0.00){\amptriangle}}
\put(0.00,0.00){\line(1,-1){36.00}}
\put(0.00,0.00){\line(0,1){19.00}}
\put(0.00,0.00){\line(-1,-1){17.00}}
\end{picture}}
}

\def\camptriangle{\thicklines
{
\unitlength=.600000pt
\begin{picture}(100.00,35.00)(-45.00,-15)
\put(3.00,16.00){\makebox(0.00,0.00){\verytriangle}}
\put(4.0,-9.00){\makebox(0.00,0.00)[l]{\scriptsize $s$}}
\put(0.00,1.00){\line(0,-1){19.00}}
\put(0.00,-19.00){\makebox(0.00,0.00){$\bullet$}}
\put(0.00,-19.00){\line(1,-2){8.50}}
\put(0.00,-19.00){\line(-1,-2){8.50}}
\put(19,-34){\makebox(0.00,0.00){\scriptsize $\cdots$}}
\put(-19,-34){\makebox(0.00,0.00){\scriptsize $\cdots$}}
\end{picture}}
}

\def\AAnulamin{\thicklines
{
\unitlength=.700000pt
\begin{picture}(10.00,15.00)(-5.00,0.00)
\put(0.00,0.00){\makebox(0.00,0.00){$\circ$}}
\put(0.00,15.00){\line(0,-1){12.00}}
\end{picture}}
}

\def\AAblack{\thicklines
{
\unitlength=.700000pt
\begin{picture}(10.00,20.00)(-5.00,0.00)
\put(0.00,0.00){\makebox(0.00,0.00){$\bullet$}}
\put(0.00,15.00){\line(0,-1){12.00}}
\end{picture}}
}

\def\AAparm#1#2{\thicklines
{
\unitlength=.500000pt
\begin{picture}(60.00,50.00)(-10,10.00)
\put(43.00,-2.00){\makebox(0.00,0.00)[t]{\scriptsize $#2$}}
\put(-3.00,-2.00){\makebox(0.00,0.00)[t]{\scriptsize $#1$}}
\put(43.00,7.00){\makebox(0.00,0.00){$\circ$}}
\put(-3.00,7.00){\makebox(0.00,0.00){$\circ$}}
\put(20.00,30.00){\makebox(0.00,0.00){$\bullet$}}
\put(20.00,30.00){\line(1,-1){20.00}}
\put(20.00,30.00){\line(-1,-1){20.00}}
\put(20.00,50.00){\line(0,-1){20.00}}
\end{picture}}
}

\def\AAAAparm#1#2#3{\thicklines
{
\unitlength=.50000pt
\begin{picture}(60.00,50.00)(-10,13.00)
\put(43.00,-2.00){\makebox(0.00,0.00)[t]{\scriptsize $#2$}}
\put(-3.00,-2.00){\makebox(0.00,0.00)[t]{\scriptsize $#1$}}
\put(30.00,30.00){\makebox(0.00,0.00)[l]{\scriptsize $#3$}}
\put(45.00,5.00){\makebox(0.00,0.00){$\circ$}}
\put(-5.00,5.00){\makebox(0.00,0.00){$\circ$}}
\put(20.00,30.00){\makebox(0.00,0.00){$\circ$}}
\put(22.00,28.00){\line(1,-1){20.00}}
\put(18.00,28.00){\line(-1,-1){20.00}}
\put(20.00,55.00){\line(0,-1){20.00}}
\end{picture}}
}

\def\trioparm#1{\thicklines
{
\unitlength=.5pt
\begin{picture}(60.00,50.00)(-10,18.00)
\put(30.00,30.00){\makebox(0.00,0.00)[l]{\scriptsize $#1$}}
\put(20.00,30.00){\makebox(0.00,0.00){$\circ$}}
\put(22.00,28.00){\line(1,-1){20.00}}
\put(18.00,28.00){\line(-1,-1){20.00}}
\put(20.00,55.00){\line(0,-1){20.00}}
\end{picture}}
}

\def\trioextparam#1{\thicklines
{
\unitlength=.5pt
\begin{picture}(60.00,50.00)(-10,40.00)
\put(30.00,50.00){\makebox(0.00,0.00)[l]{\scriptsize $#1$}}
\put(20.00,52.00){\makebox(0.00,0.00){$\circ$}}
\put(20.00,30.00){\makebox(0.00,0.00){$\bullet$}}
\put(22.00,28.00){\line(1,-1){20.00}}
\put(18.00,28.00){\line(-1,-1){20.00}}
\put(20.00,49.00){\line(0,-1){20.00}}
\put(20.00,73.00){\line(0,-1){16.00}}
\end{picture}}
}

\def\AAAparm#1#2#3{\thicklines
{
\unitlength=.500000pt
\begin{picture}(60.00,50.00)(-10,33.00)
\put(43.00,15.00){\makebox(0.00,0.00)[b]{\scriptsize $#2$}}
\put(-3.00,15.00){\makebox(0.00,0.00)[b]{\scriptsize $#1$}}
\put(27.00,52.00){\makebox(0.00,0.00)[l]{\scriptsize $#3$}}
\put(43.00,7.00){\makebox(0.00,0.00){$\circ$}}
\put(-3.00,7.00){\makebox(0.00,0.00){$\circ$}}
\put(20.00,52.00){\makebox(0.00,0.00){$\circ$}}
\put(20.00,30.00){\makebox(0.00,0.00){$\bullet$}}
\put(20.00,30.00){\line(1,-1){20.00}}
\put(20.00,30.00){\line(-1,-1){20.00}}
\put(20.00,50.00){\line(0,-1){20.00}}
\put(20.00,57.00){\line(0,1){15.00}}
\end{picture}}
}

\def\AAparm#1#2{\thicklines
{
\unitlength=.500000pt
\begin{picture}(60.00,40.00)(-10,20.00)
\put(43.00,15.00){\makebox(0.00,0.00)[b]{\scriptsize $#2$}}
\put(-3.00,15.00){\makebox(0.00,0.00)[b]{\scriptsize $#1$}}
\put(43.00,7.00){\makebox(0.00,0.00){$\circ$}}
\put(-3.00,7.00){\makebox(0.00,0.00){$\circ$}}
\put(20.00,30.00){\makebox(0.00,0.00){$\bullet$}}
\put(20.00,30.00){\line(1,-1){20.00}}
\put(20.00,30.00){\line(-1,-1){20.00}}
\put(20.00,50.00){\line(0,-1){20.00}}
\end{picture}}
}

\def\AAAl{\thicklines
{
\unitlength=.500000pt
\begin{picture}(60.00,40.00)(-10,20.00)
\put(-3.00,15.00){\makebox(0.00,0.00)[b]{\scriptsize $1$}}
\put(-3.00,7.00){\makebox(0.00,0.00){$\circ$}}
\put(20.00,30.00){\makebox(0.00,0.00){$\bullet$}}
\put(20.00,30.00){\line(1,-1){20.00}}
\put(20.00,30.00){\line(-1,-1){20.00}}
\put(20.00,50.00){\line(0,-1){20.00}}
\end{picture}}
}

\def\AAAr{\thicklines
{
\unitlength=.500000pt
\begin{picture}(60.00,40.00)(-10,20.00)
\put(43.00,15.00){\makebox(0.00,0.00)[b]{\scriptsize $1$}}
\put(43.00,7.00){\makebox(0.00,0.00){$\circ$}}
\put(20.00,30.00){\makebox(0.00,0.00){$\bullet$}}
\put(20.00,30.00){\line(1,-1){20.00}}
\put(20.00,30.00){\line(-1,-1){20.00}}
\put(20.00,50.00){\line(0,-1){20.00}}
\end{picture}}
}

\def\AAAlextparm#1{\thicklines
{
\unitlength=.500000pt
\begin{picture}(80.00,40.00)(-30,20.00)
\put(-3.00,15.00){\makebox(0.00,0.00)[b]{\scriptsize $#1$}}
\put(-3.00,7.00){\makebox(0.00,0.00){$\circ$}}
\put(20.00,30.00){\makebox(0.00,0.00){$\bullet$}}
\put(20.00,30.00){\line(1,-1){20.00}}
\put(20.00,30.00){\line(-1,-1){20.00}}
\put(-6.00,4.00){\line(-1,-1){20.00}}
\put(20.00,50.00){\line(0,-1){20.00}}
\end{picture}}
}

\def\AAArextparm#1{\thicklines
{
\unitlength=.500000pt
\begin{picture}(80.00,40.00)(0,20.00)
\put(43.00,15.00){\makebox(0.00,0.00)[b]{\scriptsize $#1$}}
\put(43.00,7.00){\makebox(0.00,0.00){$\circ$}}
\put(20.00,30.00){\makebox(0.00,0.00){$\bullet$}}
\put(20.00,30.00){\line(1,-1){20.00}}
\put(46.00,4.00){\line(1,-1){20.00}}
\put(20.00,30.00){\line(-1,-1){20.00}}
\put(20.00,50.00){\line(0,-1){20.00}}
\end{picture}}
}

\def\AAlparm#1#2{\thicklines
{
\unitlength=.500000pt
\begin{picture}(60.00,40.00)(-10,10.00)
\put(43.00,15.00){\makebox(0.00,0.00)[b]{\scriptsize $#2$}}
\put(-3.00,15.00){\makebox(0.00,0.00)[b]{\scriptsize $#1$}}
\put(43.00,7.00){\makebox(0.00,0.00){$\circ$}}
\put(-3.00,7.00){\makebox(0.00,0.00){$\circ$}}
\put(20.00,30.00){\makebox(0.00,0.00){$\bullet$}}
\put(20.00,30.00){\line(1,-1){20.00}}
\put(20.00,30.00){\line(-1,-1){20.00}}
\put(20.00,50.00){\line(0,-1){20.00}}
\put(43.00,5.00){\line(0,-1){15.00}}
\put(43.00,-10.00){\makebox(0.00,0.00){$\bullet$}}
\end{picture}}
}

\def\AAlparmm#1#2{\thicklines
{
\unitlength=.500000pt
\begin{picture}(60.00,40.00)(-10,10.00)
\put(43.00,15.00){\makebox(0.00,0.00)[b]{\scriptsize $#2$}}
\put(-3.00,15.00){\makebox(0.00,0.00)[b]{\scriptsize $#1$}}
\put(43.00,7.00){\makebox(0.00,0.00){$\circ$}}
\put(-3.00,7.00){\makebox(0.00,0.00){$\circ$}}
\put(20.00,30.00){\makebox(0.00,0.00){$\bullet$}}
\put(20.00,30.00){\line(1,-1){20.00}}
\put(20.00,30.00){\line(-1,-1){20.00}}
\put(20.00,50.00){\line(0,-1){20.00}}
\put(43.00,4.00){\line(0,-1){15.00}}
\end{picture}}
}

\def\AArparm#1#2{\thicklines
{
\unitlength=.500000pt
\begin{picture}(60.00,40.00)(-10,10.00)
\put(43.00,15.00){\makebox(0.00,0.00)[b]{\scriptsize $#2$}}
\put(-3.00,15.00){\makebox(0.00,0.00)[b]{\scriptsize $#1$}}
\put(43.00,7.00){\makebox(0.00,0.00){$\circ$}}
\put(-3.00,7.00){\makebox(0.00,0.00){$\circ$}}
\put(20.00,30.00){\makebox(0.00,0.00){$\bullet$}}
\put(20.00,30.00){\line(1,-1){20.00}}
\put(20.00,30.00){\line(-1,-1){20.00}}
\put(20.00,50.00){\line(0,-1){20.00}}
\put(-3.00,5.00){\line(0,-1){15.00}}
\put(-3.00,-10.00){\makebox(0.00,0.00){$\bullet$}}
\end{picture}}
}

\def\AAAlparm#1#2#3{\thicklines
{
\unitlength=.500000pt
\begin{picture}(60.00,40.00)(-10,10.00)
\put(43.00,15.00){\makebox(0.00,0.00)[b]{\scriptsize $#2$}}
\put(-3.00,15.00){\makebox(0.00,0.00)[b]{\scriptsize $#1$}}
\put(5.00,-15.00){\makebox(0.00,0.00)[l]{\scriptsize $#3$}}
\put(43.00,7.00){\makebox(0.00,0.00){$\circ$}}
\put(-3.00,7.00){\makebox(0.00,0.00){$\circ$}}
\put(20.00,30.00){\makebox(0.00,0.00){$\bullet$}}
\put(20.00,30.00){\line(1,-1){20.00}}
\put(20.00,30.00){\line(-1,-1){20.00}}
\put(20.00,50.00){\line(0,-1){20.00}}
\put(-3.00,5.00){\line(0,-1){15.00}}
\put(-3.00,-16.00){\makebox(0.00,0.00){$\circ$}}
\end{picture}}
}

\def\AAArparm#1#2#3{\thicklines
{
\unitlength=.500000pt
\begin{picture}(60.00,40.00)(-50,10.00)
\put(-43.00,15.00){\makebox(0.00,0.00)[b]{\scriptsize $#2$}}
\put(3.00,15.00){\makebox(0.00,0.00)[b]{\scriptsize $#1$}}
\put(-5.00,-15.00){\makebox(0.00,0.00)[r]{\scriptsize $#3$}}
\put(-43.00,7.00){\makebox(0.00,0.00){$\circ$}}
\put(3.00,7.00){\makebox(0.00,0.00){$\circ$}}
\put(-20.00,30.00){\makebox(0.00,0.00){$\bullet$}}
\put(-20.00,30.00){\line(-1,-1){20.00}}
\put(-20.00,30.00){\line(1,-1){20.00}}
\put(-20.00,50.00){\line(0,-1){20.00}}
\put(3.00,5.00){\line(0,-1){15.00}}
\put(3.00,-16.00){\makebox(0.00,0.00){$\circ$}}
\end{picture}}
}

\def\AAlrparm#1#2{\thicklines
{
\unitlength=.500000pt
\begin{picture}(60.00,40.00)(-10,10.00)
\put(43.00,15.00){\makebox(0.00,0.00)[b]{\scriptsize $#2$}}
\put(-3.00,15.00){\makebox(0.00,0.00)[b]{\scriptsize $#1$}}
\put(43.00,7.00){\makebox(0.00,0.00){$\circ$}}
\put(-3.00,7.00){\makebox(0.00,0.00){$\circ$}}
\put(20.00,30.00){\makebox(0.00,0.00){$\bullet$}}
\put(20.00,30.00){\line(1,-1){20.00}}
\put(20.00,30.00){\line(-1,-1){20.00}}
\put(20.00,50.00){\line(0,-1){20.00}}
\put(-3.00,4.00){\line(0,-1){15.00}}
\put(43.00,4.00){\line(0,-1){15.00}}
\end{picture}}
}

\def\AArparmm#1#2{\thicklines
{
\unitlength=.500000pt
\begin{picture}(60.00,40.00)(-10,10.00)
\put(43.00,15.00){\makebox(0.00,0.00)[b]{\scriptsize $#2$}}
\put(-3.00,15.00){\makebox(0.00,0.00)[b]{\scriptsize $#1$}}
\put(43.00,7.00){\makebox(0.00,0.00){$\circ$}}
\put(-3.00,7.00){\makebox(0.00,0.00){$\circ$}}
\put(20.00,30.00){\makebox(0.00,0.00){$\bullet$}}
\put(20.00,30.00){\line(1,-1){20.00}}
\put(20.00,30.00){\line(-1,-1){20.00}}
\put(20.00,50.00){\line(0,-1){20.00}}
\put(-3.00,4.00){\line(0,-1){15.00}}
\end{picture}}
}

\def\AArrparmm#1#2{\thicklines
{
\unitlength=.500000pt
\begin{picture}(60.00,40.00)(-10,10.00)
\put(43.00,15.00){\makebox(0.00,0.00)[b]{\scriptsize $#2$}}
\put(-3.00,15.00){\makebox(0.00,0.00)[b]{\scriptsize $#1$}}
\put(43.00,7.00){\makebox(0.00,0.00){$\circ$}}
\put(-3.00,7.00){\makebox(0.00,0.00){$\circ$}}
\put(20.00,30.00){\makebox(0.00,0.00){$\bullet$}}
\put(20.00,30.00){\line(1,-1){20.00}}
\put(20.00,30.00){\line(-1,-1){20.00}}
\put(20.00,50.00){\line(0,-1){20.00}}
\put(-5,3.50){\qbezier(0,0)(-3,-8)(-6,-16)}
\put(-1,3.50){\qbezier(0,0)(3,-8)(6,-16)}
\end{picture}}
}

\def\AAllparmm#1#2{\thicklines
{
\unitlength=.500000pt
\begin{picture}(60.00,40.00)(-10,10.00)
\put(43.00,15.00){\makebox(0.00,0.00)[b]{\scriptsize $#2$}}
\put(-3.00,15.00){\makebox(0.00,0.00)[b]{\scriptsize $#1$}}
\put(43.00,7.00){\makebox(0.00,0.00){$\circ$}}
\put(-3.00,7.00){\makebox(0.00,0.00){$\circ$}}
\put(20.00,30.00){\makebox(0.00,0.00){$\bullet$}}
\put(20.00,30.00){\line(1,-1){20.00}}
\put(20.00,30.00){\line(-1,-1){20.00}}
\put(20.00,50.00){\line(0,-1){20.00}}
\put(41,3.50){\qbezier(0,0)(-3,-8)(-6,-16)}
\put(45,3.50){\qbezier(0,0)(3,-8)(6,-16)}
\end{picture}}
}

\def\amputabletree#1#2{\thicklines
{
\unitlength=.500000pt
\begin{picture}(60.00,40.00)(-10,10.00)
\put(43.00,15.00){\makebox(0.00,0.00)[b]{\scriptsize $#2$}}
\put(-3.00,15.00){\makebox(0.00,0.00)[b]{\scriptsize $#1$}}
\put(43.00,7.00){\makebox(0.00,0.00){$\circ$}}
\put(-3.00,7.00){\makebox(0.00,0.00){$\circ$}}
\put(20.00,30.00){\makebox(0.00,0.00){$\bullet$}}
\put(20.00,30.00){\line(1,-1){20.00}}
\put(20.00,30.00){\line(-1,-1){20.00}}
\put(20.00,50.00){\line(0,-1){20.00}}
\put(-3.00,4.00){\line(0,-1){15.00}}
\put(-3.00,-10.00){\makebox(0.00,0.00){$\bullet$}}
\put(43.00,4.00){\line(0,-1){15.00}}
\end{picture}}
}

\def\notamputabletree{\thicklines
{
\unitlength=.500000pt
\begin{picture}(60.00,40.00)(-10,10.00)
\put(43.00,7.00){\makebox(0.00,0.00){$\circ$}}
\put(20.00,30.00){\makebox(0.00,0.00){$\bullet$}}
\put(20.00,30.00){\line(1,-1){20.00}}
\put(20.00,30.00){\line(-1,-1){20.00}}
\put(20.00,50.00){\line(0,-1){20.00}}
\put(43.00,4.00){\line(0,-1){15.00}}
\end{picture}}
}

\def\amputatedtree#1#2{\thicklines
{
\unitlength=.500000pt
\begin{picture}(60.00,40.00)(-10,10.00)
\put(43.00,15.00){\makebox(0.00,0.00)[b]{\scriptsize $#2$}}
\put(-3.00,15.00){\makebox(0.00,0.00)[b]{\scriptsize $#1$}}
\put(43.00,7.00){\makebox(0.00,0.00){$\circ$}}
\put(-3.00,7.00){\makebox(0.00,0.00){$\circ$}}
\put(20.00,30.00){\makebox(0.00,0.00){$\bullet$}}
\put(20.00,30.00){\line(1,-1){20.00}}
\put(20.00,30.00){\line(-1,-1){20.00}}
\put(20.00,50.00){\line(0,-1){20.00}}
\put(-3.00,4.00){\line(0,-1){15.00}}
\put(-3.00,-10.00){\makebox(0.00,0.00){$\bullet$}}
\end{picture}}
}

\def\VVparm#1#2{\thicklines
{
\unitlength=.500000pt
\begin{picture}(30.00,40.00)(-5.00,10.00)
\put(10.00,0.00){\makebox(0.00,0.00)[l]{\scriptsize $#2$}}
\put(10.00,20.00){\makebox(0.00,0.00)[l]{\scriptsize $#1$}}
\put(0.00,-5.00){\makebox(0.00,0.00){$\circ$}}
\put(0.00,20.00){\makebox(0.00,0.00){$\circ$}}
\put(0.00,0.00){\line(0,1){16.00}}
\put(0.00,40.00){\line(0,-1){16.00}}
\end{picture}}
}

\def\VVV{\thicklines
{
\unitlength=.500000pt
\begin{picture}(30.00,40.00)(-5.00,10.00)
\put(10.00,20.00){\makebox(0.00,0.00)[l]{\scriptsize $1$}}
\put(0.00,20.00){\makebox(0.00,0.00){$\circ$}}
\put(0.00,0.00){\line(0,1){16.00}}
\put(0.00,40.00){\line(0,-1){16.00}}
\end{picture}}
}

\begin{document}

\bibliographystyle{smfplain}

\frontmatter

\begin{abstract}
In this first paper of a series we study various operads of natural
operations on Hochschild cochains and relationships between them.
\end{abstract}

\begin{altabstract}
Dans ce premier article d'une s\'erie nous \'etudions et comparons plusieurs
op\'erades munies d'une action naturelle sur les cochaines de Hochschild
d'une
alg\`ebre associative.
\end{altabstract}


\maketitle

\tableofcontents

\mainmatter

\section{Introduction}

\noindent
This paper continues the efforts 
of~\cite{markl:CZEMJ07,batanin-markl,bb} in which we studied operads naturally
acting on Hochschild cochains of an associative or symmetric Frobenius
algebra.  A~general approach to the operads of natural operations in
algebraic categories was set up in \cite{markl:CZEMJ07} and the first
breakthrough in computing the homotopy type of such an operad has been
achieved in \cite{batanin-markl}. In \cite{bb}, the same problem was
approached from a~combinatorial point of view, and a machinery which
produces operads acting on the Hochschild cochain complex in a
general categorical setting was introduced.

The constructions of~\cite{bb} have some specific features in different
categories which are important in applications.
In this first paper of a
series entitled `Operads of Natural Operations'
we begin a detailed study of these special cases.

It is very natural to start with the classical
Hochschild cochain complex of an associative algebra. This is, by far, the most
studied case. It seems to us, however, that a systematic treatment is missing despite its long history and a vast amount of
literature available. One of the motivations of this paper
was our wish to relate various approaches in literature and to
provide a uniform combinatorial language for this purpose.

Here is a short {\bf summary\/} of the paper. 

In section~\ref{lat} we describe our main combinatorial tool: the
lattice path operad $\Latnic$ and its condensation in the differential
graded setting.  This description leads to a careful treatment of
(higher) brace operations and their relationship with lattice paths in
section~\ref{weeak}.

The lattice path operad comes equipped with a filtration by complexity
\cite{bb}. The second filtration stage $\Lat 2$ is the most important
for understanding natural operations on the Hochschild cochains.  In
section~\ref{s1} we give an alternative description of $\Lat 2$ in
terms of trees, closely related to the operad of natural operations
from \cite{markl:CZEMJ07}.  Finally, in section~\ref{brr} we study various
suboperads generated by brace operations. The main result is that {\bf all
these operads have the homotopy type of a chain
model of the little disks operad\/}.  For sake of completeness we add a
brief appendix containing an overview of some categorical
constructions used in this paper.

\noindent 
{\bf Convention.}
If not stated otherwise, 
by an {\em operad\/} we mean a classical symmetric (i.e.~with the
symmetric groups acting on its components) operad in an appropriate
symmetric monoidal category which will be obvious from the context. The same
convention is applied to coloured operads, substitudes, multitensors and
functor-operads recalled in the appendix. 
     
\noindent
{\bf Acknowledgement.}
We would like to express our thanks to the referee for carefully reading
the paper and many useful remarks and suggestions.

\section{The lattice path operad}
\label{lat}

As usual, for a non-negative integer $m$, $[m]$ denotes the
ordinal $0 < \cdots < m$. We will use the same symbol also for
the category with objects $\rada 0m$ and the unique morphism $i
\to j$ if and only if $i \leq j$. The {\em
tensor product\/} $[m]\ot[n]$ is the category freely generated by the
$(m,n)$-grid which is, by definition, the oriented graph with vertices
$(i,j)$, $0\leq i \leq m$, $0 \leq j \leq n$, and one oriented edge
$(i',j') \to (i'',j'')$ if and only if $(i'',j'') = (i' +1,j')$ 
or $(i'',j'') = (i',j' +1)$. 

Let us recall, closely following~\cite{bb}, the {\em lattice path operad\/} 
and its basic properties.
For non-negative integers $\Rada k1n,l$ and $n \in \nat$ put
\[
\Latnic(\Rada k1n;l) :=
\catstar([l+1],[k_1+1]\otimes\cdots\otimes [k_n+1])
\] 
where $\otimes$ is the tensor product recalled above and 
$\catstar([l+1],[k_1+1]\otimes\cdots\otimes [k_n+1])$ the set of
functors $\varphi$ that preserve the extremal points, by which we mean that
\begin{equation}
\label{endpots}
\varphi(0) = (\rada 00)\ \mbox { and }\ \varphi(l+1) = (\rada{k_1+1}{k_n+1}).
\end{equation}

A functor $\varphi \in \Latnic(\Rada k1n;l)$ is given by a chain
of $l+1$ morphisms $\varphi(0) \to \varphi(1) \to
\cdots \to \varphi(l+1)$ in $[k_1+1]\otimes\cdots\otimes [k_n+1]$ with
$\varphi(0)$ and $\varphi(l+1)$ fulfilling~(\ref{endpots}). 
Each morphism $\varphi(i) \to \varphi(i+1)$ is 
determined by a finite oriented edge-path in the
$(k_1+1,\ldots,k_n+1)$-grid.
For $n=0$, $\Latnic(;l)$ 
consists of the unique functor from $[l+1]$ to the terminal
category with one object.

\begin{latpath}
\label{ofinka}
We will  use a slight modification of the
terminology of~\cite{bb}.
For non-negative integers $\Rada k1n \in \nonneg$ 
denote by $\cub(\Rada k1n)$ the
integral hypercube
\[
\cub(\Rada k1n) :=[k_1+1] \times \cdots \times [k_n+1] \subset 
{\mathbb Z}^{\times n}.
\]
A {\em lattice path\/} is a sequence $p = (\Rada x1N)$  
of $N: =k_1+\cdots + k_n + n + 1$ points of $\cub(\Rada k1n)$ such that
$x_{a+1}$ is, for each $0 \leq a < N$, given by
increasing exactly one coordinate of $x_a$ by~$1$.
A~{\em marking\/} of $p$ is a function $\mu : p \to \nonneg$ that
assigns to each point $x_a$ of $p$ a non-negative number $\mu_a :=
\mu(x_a)$ such that $\sum_{a=1}^N \mu_a= l$.

We can describe functors in
$\Latnic(\Rada k1n;l)$ as  marked lattice paths $(p,\mu)$ in the
hypercube  $\cub(\Rada k1n)$. 
The marking $\mu_a = \mu(x_a)$ 
represents the number of elements of the interior $\{\rada 1l\}$ of $[l+1]$
that are mapped by $\varphi$ to the $a$th lattice point $x_a$ of $p$.
We call lattice points marked by $0$ {\em unmarked\/}
points so the set of marked points equals  $\varphi(\{\rada 1l\})$.
For example, the marked lattice path
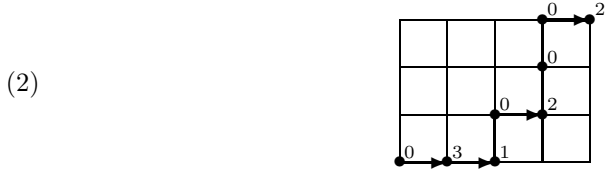
\begin{equation}
\label{vanocni_smutek}
\raisebox{-3.2em}{\rule{0pt}{0pt}}
\unitlength=1.8em
\begin{picture}(4,2)(0,1.5)
\put(0,0){\line(1,0){4}}
\put(0,1){\line(1,0){4}}
\put(0,2){\line(1,0){4}}
\put(0,3){\line(1,0){4}}
\put(0,0){\line(0,1){3}}
\put(1,0){\line(0,1){3}}
\put(2,0){\line(0,1){3}}
\put(3,0){\line(0,1){3}}
\put(4,0){\line(0,1){3}}
\thicklines
\put(0,0){\makebox(0.00,0.00){$\bullet$}}
\put(0,0){\put(.1,.1){\makebox(0.00,0.00)[lb]{\scriptsize $0$}}}
\put(0,0){\vector(1,0){1}}
\put(1,0){\makebox(0.00,0.00){$\bullet$}}
\put(1,0){\put(.1,.1){\makebox(0.00,0.00)[lb]{\scriptsize $3$}}}
\put(1,0){\vector(1,0){1}}
\put(2,0){\makebox(0.00,0.00){$\bullet$}}
\put(2,0){\put(.1,.1){\makebox(0.00,0.00)[lb]{\scriptsize $1$}}}
\put(2,0){\line(0,1){1}}
\put(2,1){\makebox(0.00,0.00){$\bullet$}}
\put(2,1){\put(.1,.1){\makebox(0.00,0.00)[lb]{\scriptsize $0$}}}
\put(2,1){\vector(1,0){1}}
\put(3,1){\makebox(0.00,0.00){$\bullet$}}
\put(3,1){\put(.1,.1){\makebox(0.00,0.00)[lb]{\scriptsize $2$}}}
\put(3,1){\line(0,1){1}}
\put(3,2){\makebox(0.00,0.00){$\bullet$}}
\put(3,2){\put(.1,.1){\makebox(0.00,0.00)[lb]{\scriptsize $0$}}}
\put(3,2){\line(0,1){1}}
\put(3,3){\makebox(0.00,0.00){$\bullet$}}
\put(3,3){\put(.1,.1){\makebox(0.00,0.00)[lb]{\scriptsize $0$}}}
\put(3,3){\vector(1,0){1}}
\put(4,3){\makebox(0.00,0.00){$\bullet$}}    
\put(4,3){\put(.1,.1){\makebox(0.00,0.00)[lb]{\scriptsize $2$}}}
\end{picture}
\end{equation}
represents a functor $\varphi \in \Latnic(3,2;8)$ with
$\varphi(0)=(0,0)$, $\varphi(1)=\varphi(2)=\varphi(3)=(1,0)$, 
$\varphi(4)=(2,0)$, $\varphi(5)= \varphi(6)=(3,1)$ and 
$\varphi(7)=\varphi(8)=\varphi(9)=(4,3)$. The lattice is trivial for
$n=0$, so the unique element of $\Latnic(;l)$ is represented 
by the point marked $l$, i.e.~by
$\bullet^l$. 
\end{latpath}

\begin{definition}
\label{Vlastik}
Let $p \in \Latnic(\Rada k1n;l)$ be a lattice path. 
A point of $p$ at which $p$ changes its
direction is an {\em angle\/} of $p$. An {\em internal point\/} of $p$ is
a point that is not an angle nor an extremal point of $p$. We
denote by $\Angl(p)$ (resp.~$\INT(p)$) the set of all angles
(resp.~internal points) of~$p$.
\end{definition}

For instance,  the path in~(\ref{vanocni_smutek})
has $4$ angles, $2$ internal points, $4$ unmarked points and $1$ unmarked
internal point.

Following again~\cite{bb} closely, we denote,
for $1\leq i < j \leq n$,  by $p_{ij}$ the projection of the
path $p \in \Latnic(\Rada k1n;l)$ 
to the face $[k_i+1]\times [k_j+1]$ of $\cub(\Rada k1n)$;  let $c_{ij} :=
\#\Angl(p_{ij})$ be the number of its angles. The maximum $c(p) :=
\max\{c_{ij}\}$ is called the {\em complexity\/} of $p$. 
Let us finally denote by $\Lat c(\Rada k1n;l) \subset 
\Latnic(\Rada k1n;l)$  the subset of marked lattice 
paths of complexity $\leq c$. The case $c=2$ is
particularly interesting, because
$\Lat 2(\Rada k1n;l)$ is, by~\cite[Proposition~2.14]{bb}, 
isomorphic to the space of 
unlabeled $(l;\Rada k1n)$-trees recalled on page~\pageref{stryc}. For
convenience of the reader we
recall this isomorphism on page~\pageref{jaruskA}.

As shown in~\cite{bb}, the sets $\Latnic(\Rada k1n;l)$ and their
subsets $\Lat c(\Rada k1n;l)$, $c \geq 0$, form an
$\nonneg$-coloured operad $\Latnic$ and its sub-operads $\Lat c$.  To
simplify formulations, we will allow $c = \infty$, putting $\Lat
\infty := \Latnic$.

\begin{convention}
Since we aim to work in the category of abelian groups, we will make 
no notational difference between the sets $\Lat c(\Rada k1n;l)$ and
their linear spans.
\end{convention}

The underlying category of the coloured operad $\Latnic$ (which
coincides with the underlying category of $\Lat c$ for any $c \geq 0$)
is, by definition, the category whose objects are non-negative
integers and morphism $n \to m$ are elements of $\Latnic(n,m)$,
i.e.~non-decreasing maps $\varphi: [m+1] \to [n+1]$ preserving the
endpoints.  

By Joyal's duality~\cite{joyal:disks}, this category is
isomorphic to the (skeletal) category $\Delta$ of finite ordered
sets, i.e.~$\Latnic(n,m) = \Delta(n,m)$.
The operadic composition makes the collection $\Lat c(\Rada
\bullet 1n;\bullet)$ (with $c=\infty$ allowed) a functor
$(\Delta^\op)^{\times n} \times \Delta \to \abel$, i.e.  $n$-times
simplicial $1$-time cosimplicial Abelian group.

Morphisms in the category $\Delta$ are generated by the cofaces $d_i : [m-1]\to [m]$ given by the non-decreasing map that misses $i$, and the codegeneracies $s_i : [m+1]\to [m]$ given by the non-decreasing map that hits $i$ twice. In both cases, $0 \leq i \leq m$. Let us inspect how these generating maps act on the pieces of the operad $\Lat c$. 

\begin{simpl}
\label{simpl}
We describe the induced $r$th ($1 \leq r \leq n$) simplicial maps
\[
\pa^r_i : \Lat c(k_1,\ldots,k_{r-1},m,k_{r+1},\ldots,k_n;l)  
\to
\Lat c(k_1,\ldots,k_{r-1},m\!-\!1,k_{r+1},\ldots,k_n;l),
\]
where $m \geq 1$, $0 \leq i \leq m$, and
\[
\sigma^r_i : \Lat c(k_1,\ldots,k_{r-1},m,k_{r+1},\ldots,k_n;l) 
\to
\Lat c(k_1,\ldots,k_{r-1},m\!+\!1,k_{r+1},\ldots,k_n;l),
\]
where $0 \leq i \leq m$.
To this end, we define, for each $m \geq 1$ and $0\leq i \leq m$, 
the epimorphism of the hypercubes
\[
D^r_i
: \cub(k_1,\ldots,k_{r-1},m,k_{r+1},\ldots,k_n) 
\twoheadrightarrow \cub(k_1,\ldots,k_{r-1},m\!-\!1,k_{r+1},\ldots,k_n)
\]
by
\[
D^r_i(a_1,\ldots,a_r,\ldots,a_n) :=
\cases{(a_1,\ldots,a_r,\ldots,a_n)}{if $a_r \leq i$, and}
{(a_1,\ldots,a_r-1,\ldots,a_n)}{if $a_r > i$,}
\]
where $(a_1,\ldots,a_r,\ldots,a_n) 
\in \cub(k_1,\ldots,k_{r-1},m,k_{r+1},\ldots,k_n)$ is an arbitrary point.
In a similar fashion, the monomorphism 
\[
S^r_i
: \cub(k_1,\ldots,k_{r-1},m,k_{r+1},\ldots,k_n) 
\hookrightarrow \cub(k_1,\ldots,k_{r-1},m\!+\!1,k_{r+1},\ldots,k_n)
\]
is, for $0\leq i \leq m$, given by
\[
S^r_i(a_1,\ldots,a_r,\ldots,a_n) :=
\cases{(a_1,\ldots,a_r,\ldots,a_n)}{if $a_r \leq i$, and}
{(a_1,\ldots,a_r+1,\ldots,a_n)}{if $a_r > i$.}
\]

Let $(p,\mu)$ be a marked lattice path in 
$\cub(k_1,\ldots,k_{r-1},m,k_{r+1},\ldots,k_n)$ representing a
functor $\varphi \in \Lat
c(k_1,\ldots,k_{r-1},m,k_{r+1},\ldots,k_n;l)$. Then $\pa^r_i(\varphi)$ is
represented by the marked path $(\pa^r_i(p),\pa^r_i(\mu))$, where
$\pa^r_i(p)$  is the image $D^r_i(p)$  of $p$ in 
$\cub(k_1,\ldots,k_{r-1},m\!-\!1,k_{r+1},\ldots,k_n,l)$. The 
marking  $\pa^r_i(\mu)$ is given by $\pa^r_i(\mu)(D^r_i(x)) :=
\sum_{\tilde x}\mu(\tilde x)$, with the sum taken over all 
$\tilde x \in p$ such that $D^r_i(\tilde x) = D^r_i(x)$. A less formal
description of this marking is the following.

There are precisely two different points of $p$, say $x'$ and
$x''$, such that $D^r_i(x') = D^r_i(x'')$; let us call the remaining
points of $p$ regular. The marking of $D^r_i(x)$ is
the same as the marking of $x$ if $x$ is regular. If $x'$ and $x''$
are the two non-regular points, then the marking of the common value
$D^r_i(x')=D^r_i(x'')$ is $\mu(x') + \mu(x'')$. See
Figure~\ref{smutek} in which the operator $\pa_1^1$ contracts the
column denoted $D^1_1$ and decorates the point 
\raisebox{.1em}{\rule{.4em}{.4em}}  obtained by
identifying the point $(1,0)$ marked $3$ with the point $(2,0)$ marked $1$ by
$3+1=4$. The remaining operators act in the similar fashion.
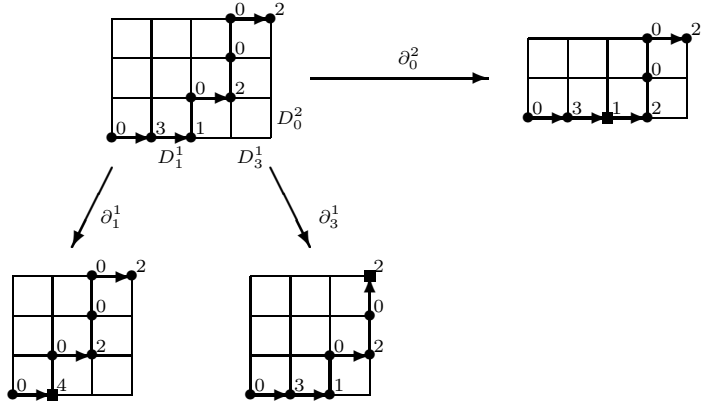
\begin{figure}[t]
\unitlength=1.5em
\begin{picture}(4,9.5)(-9.5,-6.5)
\put(-3.5,0){
\put(0,0){\line(1,0){4}}
\put(0,1){\line(1,0){4}}
\put(0,2){\line(1,0){4}}
\put(0,3){\line(1,0){4}}
\put(0,0){\line(0,1){3}}
\put(1,0){\line(0,1){3}}
\put(2,0){\line(0,1){3}}
\put(3,0){\line(0,1){3}}
\put(4,0){\line(0,1){3}}
\thicklines
\put(0,0){\makebox(0.00,0.00){$\bullet$}}
\put(0,0){\put(.1,.1){\makebox(0.00,0.00)[lb]{\scriptsize $0$}}}
\put(0,0){\vector(1,0){1}}
\put(1,0){\makebox(0.00,0.00){$\bullet$}}
\put(1,0){\put(.1,.1){\makebox(0.00,0.00)[lb]{\scriptsize $3$}}}
\put(1,0){\vector(1,0){1}}
\put(2,0){\makebox(0.00,0.00){$\bullet$}}
\put(2,0){\put(.1,.1){\makebox(0.00,0.00)[lb]{\scriptsize $1$}}}
\put(2,0){\line(0,1){1}}
\put(2,1){\makebox(0.00,0.00){$\bullet$}}
\put(2,1){\put(.1,.1){\makebox(0.00,0.00)[lb]{\scriptsize $0$}}}
\put(2,1){\vector(1,0){1}}
\put(3,1){\makebox(0.00,0.00){$\bullet$}}
\put(3,1){\put(.1,.1){\makebox(0.00,0.00)[lb]{\scriptsize $2$}}}
\put(3,1){\line(0,1){1}}
\put(3,2){\makebox(0.00,0.00){$\bullet$}}
\put(3,2){\put(.1,.1){\makebox(0.00,0.00)[lb]{\scriptsize $0$}}}
\put(3,2){\line(0,1){1}}
\put(3,3){\makebox(0.00,0.00){$\bullet$}}
\put(3,3){\put(.1,.1){\makebox(0.00,0.00)[lb]{\scriptsize $0$}}}
\put(3,3){\vector(1,0){1}}
\put(4,3){\makebox(0.00,0.00){$\bullet$}}    
\put(4,3){\put(.1,.1){\makebox(0.00,0.00)[lb]{\scriptsize $2$}}}
\put(1.5,-.5){\makebox(0.00,0.00){\scriptsize ${D^1_1}$}}
\put(3.5,-.5){\makebox(0.00,0.00){\scriptsize ${D^1_3}$}}
\put(4.5,.5){\makebox(0.00,0.00){\scriptsize ${D^2_0}$}}
}
\put(-7,-6.5){
\put(1,0){\line(1,0){3}}
\put(1,1){\line(1,0){3}}
\put(1,2){\line(1,0){3}}
\put(1,3){\line(1,0){3}}
\put(1,0){\line(0,1){3}}
\put(2,0){\line(0,1){3}}
\put(3,0){\line(0,1){3}}
\put(4,0){\line(0,1){3}}
\thicklines
\put(1,0){\makebox(0.00,0.00){$\bullet$}}
\put(1,0){\put(.1,.1){\makebox(0.00,0.00)[lb]{\scriptsize $0$}}}
\put(1,0){\vector(1,0){1}}
\put(2,0){\makebox(0.00,0.00){$\rule{.4em}{.4em}$}}
\put(2,0){\put(.1,.1){\makebox(0.00,0.00)[lb]{\scriptsize $4$}}}
\put(2,0){\line(0,1){1}}
\put(2,1){\makebox(0.00,0.00){$\bullet$}}
\put(2,1){\put(.1,.1){\makebox(0.00,0.00)[lb]{\scriptsize $0$}}}
\put(2,1){\vector(1,0){1}}
\put(3,1){\makebox(0.00,0.00){$\bullet$}}
\put(3,1){\put(.1,.1){\makebox(0.00,0.00)[lb]{\scriptsize $2$}}}
\put(3,1){\line(0,1){1}}
\put(3,2){\makebox(0.00,0.00){$\bullet$}}
\put(3,2){\put(.1,.1){\makebox(0.00,0.00)[lb]{\scriptsize $0$}}}
\put(3,2){\line(0,1){1}}
\put(3,3){\makebox(0.00,0.00){$\bullet$}}
\put(3,3){\put(.1,.1){\makebox(0.00,0.00)[lb]{\scriptsize $0$}}}
\put(3,3){\vector(1,0){1}}
\put(4,3){\makebox(0.00,0.00){$\bullet$}}    
\put(4,3){\put(.1,.1){\makebox(0.00,0.00)[lb]{\scriptsize $2$}}}
}
\put(0,-6.5){
\put(0,0){\line(1,0){3}}
\put(0,1){\line(1,0){3}}
\put(0,2){\line(1,0){3}}
\put(0,3){\line(1,0){3}}
\put(0,0){\line(0,1){3}}
\put(1,0){\line(0,1){3}}
\put(2,0){\line(0,1){3}}
\put(3,0){\line(0,1){3}}
\thicklines
\put(0,0){\makebox(0.00,0.00){$\bullet$}}
\put(0,0){\put(.1,.1){\makebox(0.00,0.00)[lb]{\scriptsize $0$}}}
\put(0,0){\vector(1,0){1}}
\put(1,0){\makebox(0.00,0.00){$\bullet$}}
\put(1,0){\put(.1,.1){\makebox(0.00,0.00)[lb]{\scriptsize $3$}}}
\put(1,0){\vector(1,0){1}}
\put(2,0){\makebox(0.00,0.00){$\bullet$}}
\put(2,0){\put(.1,.1){\makebox(0.00,0.00)[lb]{\scriptsize $1$}}}
\put(2,0){\line(0,1){1}}
\put(2,1){\makebox(0.00,0.00){$\bullet$}}
\put(2,1){\put(.1,.1){\makebox(0.00,0.00)[lb]{\scriptsize $0$}}}
\put(2,1){\vector(1,0){1}}
\put(3,1){\makebox(0.00,0.00){$\bullet$}}
\put(3,1){\put(.1,.1){\makebox(0.00,0.00)[lb]{\scriptsize $2$}}}
\put(3,1){\line(0,1){1}}
\put(3,2){\makebox(0.00,0.00){$\bullet$}}
\put(3,2){\put(.1,.1){\makebox(0.00,0.00)[lb]{\scriptsize $0$}}}
\put(3,2){\vector(0,1){1}}
\put(3,3){\makebox(0.00,0.00){$\rule{.4em}{.4em}$}}
\put(3,3){\put(.1,.1){\makebox(0.00,0.00)[lb]{\scriptsize $2$}}}
}
\put(7,-.5){
\put(0,1){\line(1,0){4}}
\put(0,2){\line(1,0){4}}
\put(0,3){\line(1,0){4}}
\put(0,1){\line(0,1){2}}
\put(1,1){\line(0,1){2}}
\put(2,1){\line(0,1){2}}
\put(3,1){\line(0,1){2}}
\put(4,1){\line(0,1){2}}
\thicklines
\put(0,1){\makebox(0.00,0.00){$\bullet$}}
\put(0,1){\put(.1,.1){\makebox(0.00,0.00)[lb]{\scriptsize $0$}}}
\put(0,1){\vector(1,0){1}}
\put(1,1){\makebox(0.00,0.00){$\bullet$}}
\put(1,1){\put(.1,.1){\makebox(0.00,0.00)[lb]{\scriptsize $3$}}}
\put(1,1){\vector(1,0){1}}
\put(2,1){\makebox(0.00,0.00){$\rule{.4em}{.4em}$}}
\put(2,1){\put(.1,.1){\makebox(0.00,0.00)[lb]{\scriptsize $1$}}}
\put(2,1){\vector(1,0){1}}
\put(3,1){\makebox(0.00,0.00){$\bullet$}}
\put(3,1){\put(.1,.1){\makebox(0.00,0.00)[lb]{\scriptsize $2$}}}
\put(3,1){\line(0,1){1}}
\put(3,2){\makebox(0.00,0.00){$\bullet$}}
\put(3,2){\put(.1,.1){\makebox(0.00,0.00)[lb]{\scriptsize $0$}}}
\put(3,2){\line(0,1){1}}
\put(3,3){\makebox(0.00,0.00){$\bullet$}}
\put(3,3){\put(.1,.1){\makebox(0.00,0.00)[lb]{\scriptsize $0$}}}
\put(3,3){\vector(1,0){1}}
\put(4,3){\makebox(0.00,0.00){$\bullet$}}    
\put(4,3){\put(.1,.1){\makebox(0.00,0.00)[lb]{\scriptsize $2$}}}
}
\thicklines
\put(-3.5,-.75){\vector(-1,-2){1}}
\put(-3.5,-2){\makebox(0.00,0.00){\scriptsize $\pa^1_1$}}   
\put(0.5,-.75){\vector(1,-2){1}}
\put(2,-2){\makebox(0.00,0.00){\scriptsize $\pa^1_3$}}   
\put(1.5,1.5){\vector(1,0){4.5}}
\put(4,2){\makebox(0.00,0.00){\scriptsize $\pa^2_0$}}   
\end{picture}
\caption{\label{smutek}
The simplicial boundaries acting on the element
of~(\ref{vanocni_smutek}).
}
\end{figure}

To define the marked lattice path $(\sigma_i^r(p),\sigma_i^r(\mu))$
representing the degeneracy $\sigma_i^r(\varphi)$, we need to observe
that the image $S^r_i(p)$ is not a lattice path in
$\cub(k_1,\ldots,k_{r-1},m\!+\!1,k_{r+1},\ldots,k_n)$, but that it can
be made one by adding a unique `missing' lattice point $\hat x$. The
resulting lattice path is $\sigma^r_i(p)$.  The marking
$\sigma_i^r(\mu)$ is given by $\sigma_i^r(\mu)(S^t_i(x)) := \mu (x)$
for $x \in p$ while $\sigma_i^r(\mu)(\hat x) := 0$, i.e.~the newly
added point $\hat x$ is unmarked. See Figure~\ref{SmuteK} in which the
new point $\hat x$ is denoted \rule{.4em}{.4em} . Observe that $\hat
x$ is always an internal~point.
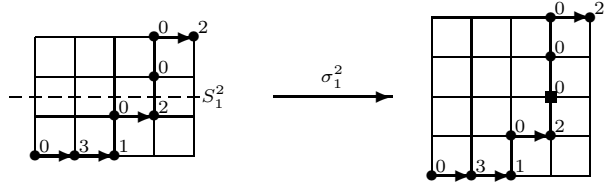
\begin{figure}[t]
\unitlength=1.5em
\begin{picture}(4,4)(-10,-.5)
\put(-5,0){
\put(0,0){\line(1,0){4}}
\put(0,1){\line(1,0){4}}
\put(0,2){\line(1,0){4}}
\put(0,3){\line(1,0){4}}
\put(0,0){\line(0,1){3}}
\put(1,0){\line(0,1){3}}
\put(2,0){\line(0,1){3}}
\put(3,0){\line(0,1){3}}
\put(4,0){\line(0,1){3}}
\thicklines
\put(0,0){\makebox(0.00,0.00){$\bullet$}}
\put(0,0){\put(.1,.1){\makebox(0.00,0.00)[lb]{\scriptsize $0$}}}
\put(0,0){\vector(1,0){1}}
\put(1,0){\makebox(0.00,0.00){$\bullet$}}
\put(1,0){\put(.1,.1){\makebox(0.00,0.00)[lb]{\scriptsize $3$}}}
\put(1,0){\vector(1,0){1}}
\put(2,0){\makebox(0.00,0.00){$\bullet$}}
\put(2,0){\put(.1,.1){\makebox(0.00,0.00)[lb]{\scriptsize $1$}}}
\put(2,0){\line(0,1){1}}
\put(2,1){\makebox(0.00,0.00){$\bullet$}}
\put(2,1){\put(.1,.1){\makebox(0.00,0.00)[lb]{\scriptsize $0$}}}
\put(2,1){\vector(1,0){1}}
\put(3,1){\makebox(0.00,0.00){$\bullet$}}
\put(3,1){\put(.1,.1){\makebox(0.00,0.00)[lb]{\scriptsize $2$}}}
\put(3,1){\line(0,1){1}}
\put(3,2){\makebox(0.00,0.00){$\bullet$}}
\put(3,2){\put(.1,.1){\makebox(0.00,0.00)[lb]{\scriptsize $0$}}}
\put(3,2){\line(0,1){1}}
\put(3,3){\makebox(0.00,0.00){$\bullet$}}
\put(3,3){\put(.1,.1){\makebox(0.00,0.00)[lb]{\scriptsize $0$}}}
\put(3,3){\vector(1,0){1}}
\put(4,3){\makebox(0.00,0.00){$\bullet$}}    
\put(4,3){\put(.1,.1){\makebox(0.00,0.00)[lb]{\scriptsize $2$}}}
\thinlines
\multiput(-.65,1.5)(.5,0){10}{\line(1,0){.3}}
\put(4.5,1.5){\makebox(0.00,0.00){\scriptsize $S^2_1$}}  
}
\put(5,-.5){
\put(0,0){\line(1,0){4}}
\put(0,1){\line(1,0){4}}
\put(0,2){\line(1,0){4}}
\put(0,3){\line(1,0){4}}
\put(0,4){\line(1,0){4}}
\put(0,0){\line(0,1){4}}
\put(1,0){\line(0,1){4}}
\put(2,0){\line(0,1){4}}
\put(3,0){\line(0,1){4}}
\put(4,0){\line(0,1){4}}
\thicklines
\put(0,0){\makebox(0.00,0.00){$\bullet$}}
\put(0,0){\put(.1,.1){\makebox(0.00,0.00)[lb]{\scriptsize $0$}}}
\put(0,0){\vector(1,0){1}}
\put(1,0){\makebox(0.00,0.00){$\bullet$}}
\put(1,0){\put(.1,.1){\makebox(0.00,0.00)[lb]{\scriptsize $3$}}}
\put(1,0){\vector(1,0){1}}
\put(2,0){\makebox(0.00,0.00){$\bullet$}}
\put(2,0){\put(.1,.1){\makebox(0.00,0.00)[lb]{\scriptsize $1$}}}
\put(2,0){\line(0,1){1}}
\put(2,1){\makebox(0.00,0.00){$\bullet$}}
\put(2,1){\put(.1,.1){\makebox(0.00,0.00)[lb]{\scriptsize $0$}}}
\put(2,1){\vector(1,0){1}}
\put(3,1){\makebox(0.00,0.00){$\bullet$}}
\put(3,1){\put(.1,.1){\makebox(0.00,0.00)[lb]{\scriptsize $2$}}}
\put(3,1){\line(0,1){1}}
\put(3,3){\makebox(0.00,0.00){$\bullet$}}
\put(3,2){\put(.1,.1){\makebox(0.00,0.00)[lb]{\scriptsize $0$}}}
\put(3,2){\line(0,1){1}}
\put(3,2){\makebox(0.00,0.00){$\rule{.4em}{.4em}$}}
\put(3,3){\put(.1,.1){\makebox(0.00,0.00)[lb]{\scriptsize $0$}}}
\put(3,3){\line(0,1){1}}
\put(3,4){\makebox(0.00,0.00){$\bullet$}}
\put(3,4){\put(.1,.1){\makebox(0.00,0.00)[lb]{\scriptsize $0$}}}
\put(3,4){\vector(1,0){1}}
\put(4,4){\makebox(0.00,0.00){$\bullet$}}    
\put(4,4){\put(.1,.1){\makebox(0.00,0.00)[lb]{\scriptsize $2$}}}
}
\thicklines
\put(1,1.5){\vector(1,0){3}}
\put(2.5,2){\makebox(0.00,0.00){\scriptsize $\sigma_1^2$}}
\end{picture}
\caption{\label{SmuteK}
The operator  $\sigma_1^2$ acting on the element
of~(\ref{vanocni_smutek}).
}
\end{figure}
\end{simpl}

\begin{cosimpl}
\label{cosimpl}
We describe, for $l \geq 1$ and $0\leq i \leq l$, the boundaries 
\[
\delta^i : \Lat c(k_1,\ldots,k_n;l-1)  
\to 
\Lat c(k_1,\ldots,k_n;l)
\]
and, for $0\leq i \leq l$, the degeneracies
\[
s^i : \Lat c(k_1,\ldots,k_n,l+1) 
\to \Lat c(k_1,\ldots,k_n;l),
\]
of the  induced cosimplicial structure. Let $(p,\mu)$ be a marked path
in $\cub(k_1,\ldots,k_n;l\mp 1)$ representing a functor $\varphi \in \Lat
c(k_1,\ldots,k_n;l\mp 1)$. Neither $\delta^i$ nor $s^i$ changes
the underlying path, so $\delta^i(\varphi)$ is represented by
$(p,\delta^i(\mu))$ and $s^i(\varphi)$ by
$(p,s^i(\mu))$.

Let $\hat x := \varphi(i)$. Then the markings $\delta^i(\mu)$ and
$s^i(\mu)$ are defined by $\delta^i(\mu)(x) = s^i(\mu)(x) = \mu(x)$
for $x \not= \hat x$, while $\delta^i(\mu)(\hat x): = \mu(\hat x) + 1$
and  $s^i(\mu)(\hat x): = \mu(\hat x) - 1$.
\end{cosimpl}

\section{Weak equivalences}
\label{weeak}

\begin{unnorm}
\label{unn}
Given an $n$-simplicial cosimplicial abelian group, i.e.~a functor 
$X: {\Delta^\op}^{\times n} \times \Delta \to \abel$, denote by
$X^\bullet_* = \uTot(X(\Rada \bullet 1n;\bullet))$ the simplicial
totalization. It is a cosimplicial dg-abelian group with components
\begin{equation}
\label{posledni_den_v_Sydney}
X^\bullet_* := \bigoplus_{* = -(k_1+\cdots+k_n)} X(\Rada k1n;\bullet)
\end{equation}
bearing the degree $+1$ differential $\pa = \pa^1 + \cdots + \pa^n$, where each
$\pa^r$ is induced from the boundaries of 
the $r$th simplicial structure in the standard
manner. We also denote by $|X|^* = \oTot(\uTot(X(\Rada \bullet1n;\bullet)))$
the cosimplicial totalization of the cosimplicial dg-abelian group
$X^\bullet_*$. It is a dg-abelian group with components
\[
|X|^* = \prod_{* = l -(k_1+\cdots+k_n)} X(\Rada k1n;l) 
= \prod_{l \geq 0} \bigoplus_{l-* = k_1 + \cdots + k_n}  X(\Rada k1n;l) 
\] 
and the degree $+1$ differential $d = \delta + \pa$, where $\pa$ is as
above and $\delta$ is the standard alternating sum of the
cosimplicial boundary operators.
\end{unnorm}

According to Appendix~\ref{s0}, the dg-abelian groups $|\Lat c|(n) := 
|\Lat c(\Rada \bullet 1n;\bullet)|$ are 
the result of condensation and, therefore, assemble, for each $c \geq
0$, into a dg-operad $|\Lat c| = \{|\Lat
c|(n)\}_{n \geq 0}$. Observe that
$|\Lat 2|$  is isomorphic to the Tamarkin-Tsygan operad $\Tam$
recalled on page~\pageref{stryc}.\footnote{Whenever we refer to
sections~\ref{s1} or~\ref{brr}, we shall keep in mind that
Convention~\ref{Jak_to_dopadne_s_Jarkou?} is used in these sections.}

Let us denote, for each $n,c \geq 0$, by $\Brac c(n)$ the simplicial
totalization of the $n$-times simplicial abelian group 
$\Lat c(\Rada \bullet 1n;0)$, that is, 
\[
\Brac c^*(n) := \bigoplus_{* = -(k_1+\cdots+k_n)}\Lat c(\Rada
k1n;0), 
\] 
with the induced differential $\pa = \pa^1 + \cdots + \pa^n$.
Elements of $\Brac c(n)$ are represented by marked lattice paths $(p,0)$
with the trivial marking $\mu=0$ (all points of $p$
are unmarked). Since the trivial marking bears no information, we will
discard it from the notation.
The {\em whiskering\/} $w :  \Brac c(n) \to |\Lat c|(n)$ is defined as
\begin{equation}
\label{JaRka}
w(p) := \prod_{s \geq 0} w_s(p),
\end{equation}
where $w_s(p) \in |\Lat c|(n)$ is the sum of all marked paths, taken
with appropriate signs, obtained from
$p$ by inserting precisely $s$ new distinct internal lattice points
marked~$1$. The origin of the signs is explained in
Proposition~\ref{ww} below. 
The action of the whiskering is illustrated in Figure~\ref{Jari}. 

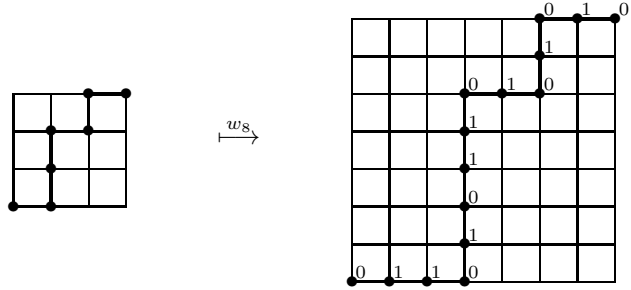
\begin{figure}
\unitlength 5mm 
\begin{picture}(7,7)(0,0)
\thinlines
\put(4,2){\line(1,0){3}}
\put(4,3){\line(1,0){3}}
\put(4,4){\line(1,0){3}}
\put(4,5){\line(1,0){3}}
\put(4,2){\line(0,1){3}}
\put(5,2){\line(0,1){3}}
\put(6,2){\line(0,1){3}}
\put(7,2){\line(0,1){3}}
\thicklines
\put(4,2){\line(1,0){1}}
\put(5,2){\line(0,1){2}}
\put(5,4){\line(1,0){1}}
\put(6,4){\line(0,1){1}}
\put(6,5){\line(1,0){1}}
\put(4,2){\makebox(0,0)[cc]{$\bullet$}}
\put(5,2){\makebox(0,0)[cc]{$\bullet$}}
\put(5,3){\makebox(0,0)[cc]{$\bullet$}}
\put(5,4){\makebox(0,0)[cc]{$\bullet$}}
\put(6,4){\makebox(0,0)[cc]{$\bullet$}}
\put(6,5){\makebox(0,0)[cc]{$\bullet$}}
\put(7,5){\makebox(0,0)[cc]{$\bullet$}}

\put(10,4){\makebox(0,0)[cc]{$\stackrel {w_8}  \longmapsto$}}

\put(8,-3){
\unitlength 5mm 
\thinlines
\put(5,3){\line(0,1){7}}
\put(6,3){\line(0,1){7}}
\put(7,3){\line(0,1){7}}
\put(8,3){\line(0,1){7}}
\put(9,3){\line(0,1){7}}
\put(10,3){\line(0,1){7}}
\put(11,3){\line(0,1){7}}
\put(12,3){\line(0,1){7}}
\put(5,3){\line(1,0){7}}
\put(5,4){\line(1,0){7}}
\put(5,5){\line(1,0){7}}
\put(5,6){\line(1,0){7}}
\put(5,7){\line(1,0){7}}
\put(5,8){\line(1,0){7}}
\put(5,9){\line(1,0){7}}
\put(5,10){\line(1,0){7}}
\thicklines
\put(5,3){\line(1,0){3}}
\put(8,3){\line(0,1){5}}
\put(8,8){\line(1,0){2}}
\put(10,8){\line(0,1){2}}
\put(10,10){\line(1,0){2}}
\put(5,3){\makebox(0,0)[cc]{$\bullet$}}
\put(6,3){\makebox(0,0)[cc]{$\bullet$}}
\put(7,3){\makebox(0,0)[cc]{$\bullet$}}
\put(8,3){\makebox(0,0)[cc]{$\bullet$}}
\put(8,4){\makebox(0,0)[cc]{$\bullet$}}
\put(8,5){\makebox(0,0)[cc]{$\bullet$}}
\put(8,6){\makebox(0,0)[cc]{$\bullet$}}
\put(8,7){\makebox(0,0)[cc]{$\bullet$}}
\put(8,8){\makebox(0,0)[cc]{$\bullet$}}
\put(9,8){\makebox(0,0)[cc]{$\bullet$}}
\put(10,8){\makebox(0,0)[cc]{$\bullet$}}
\put(10,9){\makebox(0,0)[cc]{$\bullet$}}
\put(10,10){\makebox(0,0)[cc]{$\bullet$}}
\put(11,10){\makebox(0,0)[cc]{$\bullet$}}
\put(12,10){\makebox(0,0)[cc]{$\bullet$}}
\put(0.1,0.1){
\put(5,3){\makebox(0,0)[lb]{\scriptsize$0$}}
\put(6,3){\makebox(0,0)[lb]{\scriptsize$1$}}
\put(7,3){\makebox(0,0)[lb]{\scriptsize$1$}}
\put(8,3){\makebox(0,0)[lb]{\scriptsize$0$}}
\put(8,4){\makebox(0,0)[lb]{\scriptsize$1$}}
\put(8,5){\makebox(0,0)[lb]{\scriptsize$0$}}
\put(8,6){\makebox(0,0)[lb]{\scriptsize$1$}}
\put(8,7){\makebox(0,0)[lb]{\scriptsize$1$}}
\put(8,8){\makebox(0,0)[lb]{\scriptsize$0$}}
\put(9,8){\makebox(0,0)[lb]{\scriptsize$1$}}
\put(10,8){\makebox(0,0)[lb]{\scriptsize$0$}}
\put(10,9){\makebox(0,0)[lb]{\scriptsize$1$}}
\put(10,10){\makebox(0,0)[lb]{\scriptsize$0$}}
\put(11,10){\makebox(0,0)[lb]{\scriptsize$1$}}
\put(12,10){\makebox(0,0)[lb]{\scriptsize$0$}}
}
}
\end{picture}  
\caption{
\label{Jari}
An element $p \in \Brac c(2)$ (left) and one of the terms in
$w_8(p) \in  |\Lat c|(2)$ (right). 
The newly added internal points are marked by $1$.} 
\end{figure}

For $p' \in \Lat c(\Rada a1n;0)$, $p'' \in \Lat c(\Rada b1m;0)$ and $1
\leq i \leq n$ define
\begin{equation}
\label{pozitri_domu}
p'\circ_i p'' := p' \circ_i  w_{a_i}(p'') \in \hskip -.2em
\bigoplus_{b'_1 + \cdots + b'_m = b_1 + \cdots + b_m + a_i} \hskip -3.2em
\Lat c(a_1,\sqzld,a_{i-1},{b'}_1,\sqzld,{b'}_m,a_{i+1},\sqzld,a_n;0) 
\end{equation}
where $w_{a_i}(p'')$ is the whiskering of the lattice path $p''$ by 
$a_i$ points and $\circ_i$ in the
right hand side is the operadic composition in the coloured operad $\Lat c$.
By linearity,~(\ref{pozitri_domu}) extends to 
the operation $\circ_i : \Brac c(n)
\otimes \Brac c(m) \to \Brac c(m+n-1)$.

\begin{proposition}
\label{ww}
Operations $\circ_i$ above make the collection $\Brac c =
\{\Brac c(n)\}_{n \geq 1}$ a dg-operad. The signs in~(\ref{JaRka}) can
be chosen such that the map $w : \Brac c
\hookrightarrow |\Lat c|$ is an inclusion of dg-operads.
\end{proposition}

\begin{proof}
The first part of the proposition can be verified directly.  There is
an inductive procedure to fix the signs in~(\ref{JaRka}), but we
decided not to include this clumsy and lengthy calculation here.  A
conceptual way to get the signs is to embed the dg-operad $|\Lat c|$
into the coendomorphism operad of chains on the standard simplex of 
a~sufficiently large dimension, cf.~\cite[Remark~2.20]{bb}, and to
require that the whiskering $w:\Brac c \to |\Lat c|$ induces, via the
isomorphism of Proposition~\ref{Pozitri_prileti_Jaruska.}  below, the
action of the surjection operad, with the sign convention
of~\cite[Section~2.2]{berger-fresse}.
\end{proof}

\vskip .3em
\noindent 
{\bf Remark.} 
One of the main advantages of the `operadic' sign
convention (see~\ref{Je_zdrojem_smutku.}) which we use in sections~\ref{s1}
or~\ref{brr} is that in the corresponding whiskering
formula~(\ref{Syd}) all terms, quite miracously, appear with the $+1$-signs.
\vskip .3em

So the operad structure of $\Brac c$ is induced by the operad
structure of $|\Lat c|$ and the whiskering map. Notice that
$\Brac 2$ is the brace operad $\Br$ recalled on page~\pageref{brr} and
the map $w : \Brac 2 \to |\Lat 2|$ the whiskering defined
in~(\ref{Syd}). Proposition~\ref{ww} therefore generalizes
Proposition~\ref{whisk}.

\begin{norm}
Let $X(\Rada \bullet 1n;\bullet)$ 
be an $n$-simplicial cosimplicial abelian group as in~\ref{unn}.
We will need also the traditional $n$-simplicial  
{\em normalized\/} totalization, or simplicial normalization for
short, denoted  $\overline X^\bullet_* \!=\! \uN(X(\Rada \bullet
1n;\bullet))$, obtained from the un-normalized 
totalization~(\ref{posledni_den_v_Sydney}) by modding out the images 
of simplicial degeneracies.
We then denote by $|\overline X|^* = \oN(\uN(X(\Rada \bullet1n;\bullet)))$
the normalized cosimplicial totalization of the cosimplicial dg-abelian group
$\overline X^\bullet_*$. It is the intersection of the kernels of
cosimplicial degeneracies in the un-normalized cosimplicial
totalization of $\overline X^\bullet_*$.  
As argued in~\cite{bb}, the $n$-simplicial cosimplicial normalization   
$|\nLat c|$ of the lattice path operad $\Lat c$ is a dg-operad.  
\end{norm}

Let us denote, for each $n,c \geq 0$, by $\nBrac c(n) = 
\uN(\Lat c(\Rada \bullet 1n;0))$ the
simplicial normalization 
of the $n$-simplicial abelian group 
$\Lat c(\Rada \bullet 1n;0)$, with the induced differential.
The explicit description of the simplicial structure in~\ref{simpl}
makes it obvious that
elements of $\nBrac c(n)$ are represented by (unmarked) lattice 
paths with no internal points.

One defines the operadic composition on $\nBrac c =
\{\nBrac c(n)\}_{n \geq 0}$ and the whiskering $w : \nBrac c
\hookrightarrow |\nLat c|$ by the same formulas as in the
un-normalized case.  The operad $\nBrac 2$ is the normalized brace
operad $\NBr$ recalled on page~\pageref{brr-norm}.  We leave as an exercise
to verify that $\nBrac 1$ is the operad for unital associative
algebras and $\nBrac 0$ the operad whose `algebras' are abelian
groups with a distinguished point.

\begin{proposition}
\label{Pozitri_prileti_Jaruska.}
The operads $\nBrac c$ are isomorphic to the suboperads $F_c{\mathcal X}$ of the
{\em surjection operad\/} ${\mathcal X}$ introduced
in~\cite[1.6.2]{berger-fresse}, resp.~the suboperads ${\mathcal
S}_c$ of the {\em sequence operad\/} ${\mathcal S}$ introduced
in~\cite[Definition~3.2]{mcclure-smith:JAMS03}. 
\end{proposition}

\begin{proof}
We rely on the terminology of~\cite[1.6.2]{berger-fresse}.  A
non-degenerate surjection $u: \{\rada 1m\} \to \{\rada 1n\}$, $m \geq
n$, in $F_c{\mathcal X}(n)$ induces a lattice path $\varphi_u$
representing an element of $\nBrac c(n)$ as follows. For $1 \leq i
\leq n$ denote by $d_i \in {\mathbb Z}^{\times n}$ the vector $
(0,\ldots,1,\ldots,0)$ with $1$ at the $i$th position, and $k_i := \#
u^{-1}(i)-1$. Then $\varphi_u$ is the path in the grid $[k_1+1]\ot
\cdots \ot [k_n+1]$ that starts at the `lower left corner' $(\rada
00)$, advances by $d_{u(1)}$, then by $d_{u(2)}$, etc., and finally by
$d_{u(m)}$. It is obvious that the correspondence $u \mapsto
\varphi_u$ is one-to-one.
\end{proof}

The following statement follows from~\cite[Examples~3.10(c)]{bb} 
and~\cite[Section~1.2]{berger-fresse}.

\begin{proposition}
\label{nww}
The whiskering $w : \nBrac c \hookrightarrow |\nLat c|$ is an 
inclusion of dg-operads.
\end{proposition}

We will need also the following statement.

\begin{proposition}
\label{projj}
The natural projection $\pi : \Brac c \epi \nBrac c$ to the
normalization is an epimorphism of dg-operads for each $c \geq 0$.
\end{proposition}

\begin{proof}
It is almost obvious that the operadic composition in $\Brac c$
preserves the number of internal points, that is, if $p'$
(resp.~$p''$) is a lattice path with $a'$ (resp.~$a''$) internal
points, then $p'\circ_i p''$ is, for each $i$ for which this
expression makes sense, a~linear combination of lattice paths with $a'
+ a''$ internal points.  This implies that the degenerate subspace
$\Deg(\Brac c)$ of $\Brac c$ which is the subcollection spanned by
lattice paths with at least one internal point, form a dg-operadic
ideal in $\Brac c$, so the projection $\pi: \Brac c \epi \Brac c/
\Deg(\Brac c) = \nBrac c$ is an operad map. The fact that $\pi$
commutes with the differentials follows from the standard properties
of the simplicial normalizations.
\end{proof}

Let $\HBrac c = \{\HBrac c(n)\}_{n \geq 0}$ 
be the subcollection of $\Brac c$ such that
$\HBrac c (n)\subset \Brac c(n)$ is spanned by
paths with no internal points, for $n \geq 1$, and $\HBrac c (0) := 0$.

\begin{proposition}
The collection $\HBrac c$ is a (non-dg) suboperad of $\Brac c$ for any
$c\geq 0$. It is dg-closed if and only if $c \leq 2$.
\end{proposition}

\begin{proof}
It follows from the property stated in the proof of
Proposition~\ref{projj} that the subcollection $\HBrac c$ is closed
under the operad structure of $\Brac c$ for an arbitrary $c \geq
0$. It remains to prove that $\HBrac c$ is closed under the
action of the differential if and only if $c \leq 2$. Let us prove
first that it is dg-closed for $c \leq 2$.

For $c=2$ this follows from the fact that $\HBrac 2 = \HBr$ is a
dg-suboperad of $\Brac 2 = \Br$, see Proposition~\ref{JaRkA} and the
description of the dg-operad structures of $\Br$ and $\HBr$ 
in terms of trees following
that proposition, or~\cite{mcclure-smith}. 
For $c=0,1$, the proposition is~obvious.

If $c \geq 3$, the differential may create
internal points, as shown in the following
picture where the piece $\pa^0_1$ of the differential creates the
internal point \raisebox{.1em}{\rule{.4em}{.4em}}~:
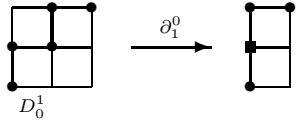
\[
\unitlength=1.5em
\begin{picture}(4,2.5)(0,-.2)
\put(-2.5,0){
\put(0,0){\line(1,0){2}}
\put(0,1){\line(1,0){2}}
\put(0,2){\line(1,0){2}}
\put(0,0){\line(0,1){2}}
\put(1,0){\line(0,1){2}}
\put(2,0){\line(0,1){2}}

\thicklines
\put(0,0){\makebox(0.00,0.00){$\bullet$}}
\put(0,0){\line(0,1){1}}
\put(0,1){\makebox(0.00,0.00){$\bullet$}}
\put(0,1){\line(1,0){1}}
\put(1,1){\makebox(0.00,0.00){$\bullet$}}
\put(1,1){\line(0,1){1}}
\put(1,2){\makebox(0.00,0.00){$\bullet$}}
\put(1,2){\line(1,0){1}}
\put(2,2){\makebox(0.00,0.00){$\bullet$}}
\put(.5,-.5){\makebox(0.00,0.00){\scriptsize $D^1_0$}}
}
\put(2.5,0){
\put(1,0){\line(1,0){1}}
\put(1,1){\line(1,0){1}}
\put(1,2){\line(1,0){1}}
\put(1,0){\line(0,1){2}}
\put(2,0){\line(0,1){2}}
\thicklines
\put(1,0){\makebox(0.00,0.00){$\bullet$}}
\put(1,0){\line(0,1){1}}
\put(1,1){\makebox(0.00,0.00){$\rule{.4em}{.4em}$}}
\put(1,1){\line(0,1){1}}
\put(1,2){\makebox(0.00,0.00){$\bullet$}}
\put(1,2){\line(1,0){1}}
\put(2,2){\makebox(0.00,0.00){$\bullet$}}
}
\thicklines
\put(.5,1){\vector(1,0){2}}
\put(1.5,1.5){\makebox(0.00,0.00){\scriptsize $\pa_1^0$}}
\end{picture}
\]

\noindent 
So $\HBrac c$ is not dg-closed if $c \geq 3$.
\end{proof}

\begin{semi}
For each $n,c\geq 0$, one may also consider the collection $|\sLat c| :=
\{|\sLat c|(n)\}_{n \geq 0}$ defined by 
\[
|\sLat c|(n) := \oTot(\uN(\Lat c(\Rada \bullet1n;\bullet)))
\]
i.e.~as the $n$-simplicial normalization followed by the
un-normalized cosimplicial totalization. 
\end{semi}

Observe that there is a natural projection $\pi: |\Lat c| \epi |\sLat
c|$ of collections.
We emphasize that, 
for $c \geq 3$, the collection $|\sLat c|$ has \underline{no} natural 
dg-operad structure although it will still play an important auxiliary
role in this section. We, however,~have

\begin{proposition}
For $c \leq 2$, the collection $|\sLat c|$ has a natural operad
structure such that the projection $\pi: |\Lat c| \epi |\sLat
c|$ is a map of dg-operads.
\end{proposition}

\begin{proof}
The proof uses the fact that $|\sLat 2|$ is 
the normalized Tamarkin-Tsygan operad $\NormT$ which is a quotient of
$\Tam = |\Lat 2|$, see~\ref{TTTT}. This proves the proposition for
$c=2$. For $c=0,1$ the claim is obvious.
\end{proof}

The main theorem of this section reads:

\begin{theorem}
\label{jArKa}
For each $c \geq 0$, 
there is the following chain of weak equivalences of dg-operads:
\[
|\Lat c| \stackrel w\longleftarrow \Brac c \stackrel\pi\longrightarrow 
\nBrac c \stackrel w\longrightarrow |\nLat c|, 
\]
in which the maps $w$ are the whiskerings of Propositions~\ref{ww}
and~\ref{nww}, and $\pi$ is the normalization projection of
Proposition~\ref{projj}.  
\end{theorem}

\begin{proof}
The map $\pi : \Brac c(n) \to \nBrac c(n)$ is a homology isomorphism
for each $n,c \geq 0$ 
because it is the normalization map of an $n$-simplicial abelian
group, so $\pi$ is a weak equivalence of dg operads. 

Let us analyze the un-normalized whiskering $w : \Brac c (n)
\hookrightarrow |\Lat c|(n)$. 
The arity $n$ piece of the dg-operad 
$|\Lat c|$ can be organized into the bicomplex of
Figure~\ref{JJJ} in which the $l$th column ${\Lat c}(n)^l_*$, 
$l \geq 0$, is the 
simplicial totalization
$\uTot(\Lat c(\Rada \bullet 1n;l))$ and the horizontal
differentials are induced from the cosimplicial structure. 
\begin{figure}
{
\unitlength=1.2pt
\begin{picture}(200.00,143.00)(-50.00,50.00)
\put(150.00,52){\makebox(0.00,0.00){$\vdots$}}
\put(100.00,52){\makebox(0.00,0.00){$\vdots$}}
\put(50.00,52){\makebox(0.00,0.00){$\vdots$}}
\put(200.00,80.00){\makebox(0.00,0.00){$\cdots$}}
\put(200.00,110.00){\makebox(0.00,0.00){$\cdots$}}
\put(200.00,140.00){\makebox(0.00,0.00){$\cdots$}}
\put(200.00,170.00){\makebox(0.00,0.00){$\cdots$}}
\put(200.00,200.00){\makebox(0.00,0.00){$\cdots$}}
\put(150.00,120.00){\vector(0,1){10.00}}
\put(100.00,90.00){\vector(0,1){10.00}}
\put(50.00,90.00){\vector(0,1){10.00}}
\put(150.00,90.00){\vector(0,1){10.00}}
\put(100.00,60.00){\vector(0,1){10.00}}
\put(50.00,60.00){\vector(0,1){10.00}}
\put(150.00,60.00){\vector(0,1){10.00}}
\put(100.00,120.00){\vector(0,1){10.00}}
\put(50.00,120.00){\vector(0,1){10.00}}
\put(150.00,150.00){\vector(0,1){10.00}}
\put(100.00,150.00){\vector(0,1){10.00}}
\put(50.00,150.00){\vector(0,1){10.00}}
\put(150.00,180.00){\vector(0,1){10.00}}
\put(100.00,180.00){\vector(0,1){10.00}}
\put(50.00,180.00){\vector(0,1){10.00}}
\put(168.00,140.00){\vector(1,0){22.00}}
\put(118.00,140.00){\vector(1,0){12.00}}
\put(68.00,140.00){\vector(1,0){12.00}}
\put(10.00,140.00){\vector(1,0){22.00}}
\put(168.00,110.00){\vector(1,0){22.00}}
\put(118.00,110.00){\vector(1,0){12.00}}
\put(68.00,110.00){\vector(1,0){12.00}}
\put(10.00,110.00){\vector(1,0){22.00}}
\put(168.00,80.00){\vector(1,0){22.00}}
\put(118.00,80.00){\vector(1,0){12.00}}
\put(68.00,80.00){\vector(1,0){12.00}}
\put(10.00,80.00){\vector(1,0){22.00}}
\put(10.00,170.00){\vector(1,0){25.00}}
\put(165.00,170.00){\vector(1,0){25.00}}
\put(115.00,170.00){\vector(1,0){20.00}}
\put(65.00,170.00){\vector(1,0){20.00}}
\put(160.00,200.00){\vector(1,0){30.00}}
\put(110.00,200.00){\vector(1,0){30.00}}
\put(60.00,200.00){\vector(1,0){30.00}}
\put(150.00,80.00){\makebox(0.00,0.00){${\Lat c}(n)^2_{-3}$}}
\put(150.00,110.00){\makebox(0.00,0.00){${\Lat c}(n)^2_{-2}$}}
\put(150.00,140.00){\makebox(0.00,0.00){${\Lat c}(n)^2_{-1}$}}
\put(150.00,170.00){\makebox(0.00,0.00){${\Lat c}(n)^2_{0}$}}
\put(100.00,80.00){\makebox(0.00,0.00){${\Lat c}(n)^1_{-3}$}}
\put(100.00,110.00){\makebox(0.00,0.00){${\Lat c}(n)^1_{-2}$}}
\put(100.00,140.00){\makebox(0.00,0.00){${\Lat c}(n)^1_{-1}$}}
\put(100.00,170.00){\makebox(0.00,0.00){${\Lat c}(n)^1_{0}$}}
\put(50.00,80.00){\makebox(0.00,0.00){${\Lat c}(n)^0_{-3}$}}
\put(50.00,110.00){\makebox(0.00,0.00){${\Lat c}(n)^0_{-2}$}}
\put(50.00,140.00){\makebox(0.00,0.00){${\Lat c}(n)^0_{-1}$}}
\put(50.00,170.00){\makebox(0.00,0.00){${\Lat c}(n)^0_{0}$}}
\put(0.00,80.00){\makebox(0.00,0.00){$0$}}
\put(0.00,110.00){\makebox(0.00,0.00){$0$}}
\put(0.00,140.00){\makebox(0.00,0.00){$0$}}
\put(0.00,170.00){\makebox(0.00,0.00){$0$}}
\put(150.00,200.00){\makebox(0.00,0.00){$0$}}
\put(100.00,200.00){\makebox(0.00,0.00){$0$}}
\put(50.00,200.00){\makebox(0.00,0.00){$0$}}
\end{picture}}
\caption{\label{Zase_cekam_jestli_Jarka_zavola.}
The structure of the dg-operad $|\Lat c|$.\label{JJJ}}
\end{figure}
The dg-abelian group $|\Lat c|(n)$ 
is then the 
corresponding ${\rm Tot}^{\prod}$-total complex
(see~\cite[Section~5.6]{weibel} for the terminology). 

The dg-abelian group $\Brac c(n)$ appears as the leftmost column of
Figure~\ref{JJJ}, so one has the projection $\proj : |\Lat c|(n) \to
\Brac c(n)$ of dg-abelian groups which is the identity on the leftmost
column and sends the remaining columns to~$0$. Since clearly $\proj
\circ w = \id$, it is enough to prove that $\proj$ is a homology isomorphism.

We interpret $\proj :  |\Lat c|(n) \to \Brac c(n)$ as a map of
bicomplexes, with $\Brac c(n)$ consisting of one
column, and we prove that $\proj$ induces an isomorphism of the $E^2$-terms
of the spectral sequences induced by the column filtrations. 
These filtrations are complete and exhaustive, thus the
Eilenberg-Moore comparison theorem~\cite[Theorem~5.5.1]{weibel}
implies that $\proj$ is a homology isomorphism.

Let $(E^0_{**},d^0)$ be the $0$th term of column spectral sequence for
$|\Lat c|(n)$. This means that $(E^0_{l,*},d^0) = (\uTot(\Lat c(\Rada
\bullet 1n;l))_*,\pa)$, the $l$th column of the bicomplex in
Figure~\ref{JJJ} with the simplicial differential. 

To calculate
$E^1_{l*}:= H_*(E^0_{l*},d^0)$, we recall the explicit
description of the simplicial structures given in~\ref{simpl} and observe
that the vertical differential $d^0=\pa$ does not increase the number of
angles of lattice paths.
We therefore have, for each {\em fixed\/} $l \geq 0$, another spectral sequence
$(\EE_{**}^r,\dd^r)$ induced by the filtration of ${\Lat c}(n)^l_*$ by
the number of angles. The piece $\EE^0_{uv}$ of the initial sheet of
this spectral sequence is spanned by marked paths $(p,\mu) \in \Lat
c(\Rada k1n;l)$ with $-u$ angles and $v = -u- (k_1 + \cdots +
k_n)$. With this degree convention, the total
degree of an element of $\EE_{**}$ is the same as the degree of
the corresponding element in $E^1_{l*}$. By simple combinatorics,
$(\EE_{**}^r,\dd^r)$ is a~spectral sequence concentrated at the region
$\{(u,v);\ u \leq 1-n,\ u-v \geq 2 -2n\}$ of the $(u,v)$-plane, thus no
convergence problems occur. One easily sees that, as
dg-abelian groups,
\begin{equation}
\label{konec_pobytu}
(\EE_{u*}^0,\dd^0) \cong
\left(\rule{0pt}{1.4em}\right. \hskip -.5em
\bigoplus_\doubless{p \in \nBrac c(n)}{\#Angl(p) = -u} 
\bigoplus_{i_1 + \cdots + i_{u+1} = n-1 -*}   
\{
\underbrace{\BBar_{i_1}\otimes_\A \cdots \otimes_\A 
\BBar_{i_{u+1}}}_{\mbox {\scriptsize $-u+1$ factors}}
\}_l,d_\BBar
\left.\rule{0pt}{1.4em}\right),
\end{equation}
where $\BBar_* = \BBar_*(\A,\A,\A)$ is the un-normalized two-sided bar
construction of the polynomial algebra $\A$ and the differential
$d_\BBar$ is induced in the standard manner from the bar
differential. 
The subscript $l$ in~(\ref{konec_pobytu}) 
denotes the $l$-homogeneous part with respect to the grading induced
by the number of instances of $x$. The factors of the direct sum are
indexed by unmarked paths with no internal points representing a basis
of $\nBrac c(n)$. The isomorphism~(\ref{konec_pobytu}) is best
explained by looking at the marked path
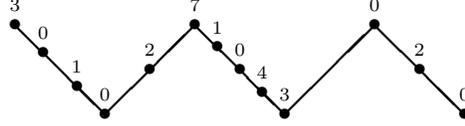
\[
{
\thicklines
\unitlength=.85pt
\begin{picture}(200.00,50.00)(0.00,0.00)
\put(200.00,0.00){\makebox(0.00,0.00){\raisebox{1.5em}{\scriptsize $0$}}}
\put(180.00,20.00){\makebox(0.00,0.00){\raisebox{1.5em}{\scriptsize $2$}}}
\put(160.00,40.00){\makebox(0.00,0.00){\raisebox{1.5em}{\scriptsize $0$}}}
\put(120.00,0.00){\makebox(0.00,0.00){\raisebox{1.5em}{\scriptsize $3$}}}
\put(110.00,10.00){\makebox(0.00,0.00){\raisebox{1.5em}{\scriptsize $4$}}}
\put(100.00,20.00){\makebox(0.00,0.00){\raisebox{1.5em}{\scriptsize $0$}}}
\put(90.00,30.00){\makebox(0.00,0.00){\raisebox{1.5em}{\scriptsize $1$}}}
\put(80.00,40.00){\makebox(0.00,0.00){\raisebox{1.5em}{\scriptsize $7$}}}
\put(60.00,20.00){\makebox(0.00,0.00){\raisebox{1.5em}{\scriptsize $2$}}}
\put(40.00,0.00){\makebox(0.00,0.00){\raisebox{1.5em}{\scriptsize $0$}}}
\put(27.50,12.50){\makebox(0.00,0.00){\raisebox{1.5em}{\scriptsize $1$}}}
\put(12.50,27.50){\makebox(0.00,0.00){\raisebox{1.5em}{\scriptsize $0$}}}
\put(0.00,40.00){\makebox(0.00,0.00){\raisebox{1.5em}{\scriptsize $3$}}}
\put(200.00,0.00){\makebox(0.00,0.00){$\bullet$}}
\put(180.00,20.00){\makebox(0.00,0.00){$\bullet$}}
\put(160.00,40.00){\makebox(0.00,0.00){$\bullet$}}
\put(120.00,0.00){\makebox(0.00,0.00){$\bullet$}}
\put(110.00,10.00){\makebox(0.00,0.00){$\bullet$}}
\put(100.00,20.00){\makebox(0.00,0.00){$\bullet$}}
\put(90.00,30.00){\makebox(0.00,0.00){$\bullet$}}
\put(80.00,40.00){\makebox(0.00,0.00){$\bullet$}}
\put(60.00,20.00){\makebox(0.00,0.00){$\bullet$}}
\put(40.00,0.00){\makebox(0.00,0.00){$\bullet$}}
\put(27.50,12.50){\makebox(0.00,0.00){$\bullet$}}
\put(12.50,27.50){\makebox(0.00,0.00){$\bullet$}}
\put(0.00,40.00){\makebox(0.00,0.00){$\bullet$}}
\put(160.00,40.00){\line(1,-1){40.00}}
\put(120.00,0.00){\line(1,1){40.00}}
\put(80.00,40.00){\line(1,-1){40.00}}
\put(40.00,0.00){\line(1,1){40.00}}
\put(0.00,40.00){\line(1,-1){40.00}}
\end{picture}}
\] 
with $4$ angles which is an element of $\EE^0_{-4,-6}$ represented, via the
isomorphism~(\ref{konec_pobytu}), by the element
\[
x^3 \ot [x^0|x^1]\ot x^0\ot [x^2] \ot x^7 \ot [x^1|x^0|x^4] \ot x^3
\ot [\ ] \ot x^0 \ot [x^2]\ot x^0
\]
in $\BBar_2 \ot_\A  \BBar_1 \ot_\A \BBar_3 \ot_\A \BBar_0 \ot_\A \BBar_1$.
It is a standard result of homological algebra that $(\BBar_* \ot_\A
\cdots  \ot_\A \BBar_*, d_\BBar)$ is acyclic in positive dimensions,
thus the cohomology of the right hand side of~(\ref{konec_pobytu}) is 
spanned by cycles of the form 
\begin{equation}
\label{zavola_Jarka?}
x^l \otimes [\ ] \otimes \cdots \otimes [\ ] \in \EE^0_{-\#\Angl(p),n-1}.
\end{equation}

At this point we need to observe that the differential $\pa$ decreases the
number of angles of lattice paths $p$ with no internal points representing
elements of $\nBrac c(n)$ by one. Indeed, it is easy to see that a
simplicial boundary operator described in~\ref{simpl} may either
decrease the number of angles of $p$ by $1$ or by $2$. When it decreases
it by $2$ it creates an internal point, so the contributions of
all simplicial boundaries that decrease the number of angles by $2$
sum up to $0$, by the standard property of the simplicial normalization. 
We conclude that $(\bigoplus_{* = u+v}\EE^1_{uv},\dd^1) 
\cong ({\nBrac c}_* (n),\pa)$ as
dg-abelian groups and that $(\EE_{**}^r,\dd^r)$ collapses at
this level.

Let us return to the column spectral sequence $(E^r_{**},d^r)$
for the bicomplex in Figure~\ref{Zase_cekam_jestli_Jarka_zavola.}.
It follows from the above calculation that the $l$th column $E^1_{l*}$ of 
the first term $(E^1_{**},d^1)$ equals $H_*(\nBrac c(n))$ for each $l
\geq 0$. It remains to describe the differential 
$d^1 : E^1_{l*} \to E^1_{(l+1)*}$. To this end, one needs to observe
that the expressions~(\ref{zavola_Jarka?}) representing
elements of $E^1_{l*} = H_*(\nBrac c(n))$ correspond to marked
lattice paths without internal points, whose only
marked point is the initial one, marked by $l$. From the
description of the cosimplicial structure given in~\ref{cosimpl}
one easily obtains that
\[
d^1 : E^1_{l*} \to E^1_{(l+1)*} = \cases{0}{if $l$ is even and}
{\id}{if $l$ is odd.}
\]
We conclude that $E^2_{**} := H_*(E^1_{**},d^1)$ is concentrated at
the leftmost column which equals $H_*(\nBrac c(n))$ and that, from the
obvious degree reasons, the column spectral sequence collapses at this
stage.  Since we already know that the projection $\Brac c
\stackrel\pi\epi \nBrac c$ is a weak equivalence~i.e., in particular,
that $H_*(\Brac c (n)) \cong H_*(\nBrac c (n))$, the above facts imply
that $\proj : |\Lat c|(n) \to \Brac c(n)$ induces an isomorphism of
the $E^2$-terms of the column spectral sequences, so it is a homology
isomorphism and $w$ is a homology isomorphism, too.

Let us finally prove that the normalized  whiskering $w : \nBrac c (n)
\hookrightarrow |\nLat c|(n)$ is a weak equivalence. We have the
composition
\begin{equation}
\label{ceka_mne_Jaruska?}
\nBrac c (n) \stackrel w\hookrightarrow |\nLat c|(n)  
\stackrel{\dot \iota}\hookrightarrow |\sLat c|(n) 
\end{equation}
in which the obvious inclusion $\dot \iota$ is a homology
isomorphism by a simple lemma formulated below. As in the
un-normalized case, the dg-abelian group $\nBrac c (n)$ is the first
column of the semi-normalized version of the bicomplex in
Figure~\ref{JJJ}, so there is a natural projection $\proj :  |\sLat
c|(n) \to \nBrac c (n)$. This $\proj$ is a homology
isomorphism by the same arguments as in the un-normalized case, only
using in~(\ref{konec_pobytu}) the normalized bar construction
instead. The proof is finished by observing that $\proj$ is the left
inverse of the composition~(\ref{ceka_mne_Jaruska?}).
\end{proof}

In the proof of Theorem~\ref{jArKa} we used the following

\begin{lemma}
\label{JARKa}
The inclusion ${\dot \iota} : |\nLat c|(n) \hookrightarrow |\sLat c|(n)$
is a homology isomorphism for each \hbox{$n,c \geq 0$}.
\end{lemma}

\begin{proof}
The lemma follows from the fact that $|\nLat c|(n)$ is the cosimplicial
normalization of the dg-cosimplicial group  $|\sLat c|(n)$.
\end{proof}

\begin{center}
-- -- -- -- --
\end{center}

In the following two sections we consider several operads.
To simplify the navigation, we give a glossary
of notation.\label{Smutek_z_Jarky.}
\[
\def\arraystretch{1.2}
\begin{array}{lll}
\Big,&\mbox {big operad of all natural operations,}
&\mbox{ page~\pageref{bIG}}
\\
\NormB,&\mbox {normalized  big operad,}
&\mbox{ page~\pageref{Jary}}
\\
\HB,&\mbox {non-unital  big operad,}
&\mbox{ page~\pageref{varr}}
\\
\Tam,&\mbox {Tamarkin-Tsygan suboperad of $\Big$,}
&\mbox{ page~\pageref{TTTT}}
\\
\NormT,&\mbox {normalized Tamarkin-Tsygan operad,}
&\mbox{ page~\pageref{tammm}}
\\
\nTam,&\mbox {non-unital Tamarkin-Tsygan operad,}
&\mbox{ page~\pageref{tammm-non}}
\\
\Br,&\mbox {brace operad,}
&\mbox{ page~\pageref{brr}}
\\
\NBr,&\mbox{normalized brace operad,}
&\mbox{ page~\pageref{brr-norm}}
\\
\HBr,&\mbox{non-unital brace operad,}
&\mbox{ page~\pageref{brr-non}}
\end{array}
\]
The operads mentioned in the list and their maps 
are organized in Figure~\ref{fig3} on page~\pageref{fig3}.

\section{Operads of natural operations}
\label{s1}

In the previous sections we studied versions of the lattice
path operad and its suboperads. We only briefly mentioned that some
of these operads act on the Hochschild cochain complex of an
associative algebra. The present and the
following sections will be devoted to this action. It
turns out that, in order to retain some nice features of the
constructions in the previous section, namely the `whiskering'
formula~(\ref{JaRka}) without signs, on one hand, and to have simple
rules for the signs in formulas for natural operations on the other
hand, one needs to use the `operadic'
degree convention, recalled in the next subsection.

\begin{cvo}
\label{Je_zdrojem_smutku.}
There are two conventions in defining the Hoch\-schild cohomology of an
associative algebra $A$. The {\em classical one\/} used for instance
in~\cite{gerstenhaber:AM63} is based on the chain
complex $\CHclas * = \bigoplus_{n \ge 0} \CHclas n$, where $\CHclas n
:= \Lin(\otexp An,A)$ (the subscript {\em cl\/} refers to
``classical''). Another appropriate name would be the {\em
(co)simplicial\/} convention, because $\CHclas *$ is a natural
cosimplicial abelian group.  With this convention, the cup product $\cup$ is a
degree $0$ operation and the Gerstenhaber bracket $[-,-]$ has degree
$-1$, see~\cite[Section~7]{gerstenhaber:AM63} for the `classical'
definitions of these operations.

On the other hand, it is typical for this part of mathematics that signs
are difficult to handle. A systematic way to control
them is the {\em Koszul sign rule\/} requiring 
that whenever we interchange two
``things'' of odd degrees, we multiply the sign by $-1$. This, however,
needs the definition
of the Hochschild cohomology as
the operadic cohomology of associative algebras~\cite{fox-markl:ContM97}. 
Now the underlying chain complex is 
\begin{equation}
\label{cano}
\CH * := \Lin(\bbT(\desusp A),\desusp A)^*, 
\end{equation}
where $\da$ denotes the desuspension of a (graded) vector
space and $\bbT(\hbox{$\desusp A$})$ the tensor algebra generated by 
$A$ placed in degree $-1$.
Explicitly,  $\CH * = \bigoplus_{n \ge -1} \CH n$, where $\CH n := 
\Lin(\otexp A{n+1},A)$, so $\CH n =  \CHclas{n+1}$ for $n \geq -1$.
With this convention, the cup product has degree 
$+1$ and the Gerstenhaber bracket degree~$0$.

Depending on the choice of the convention, there are two definitions
of the `big' operad of natural operations, see~\ref{bIG} below.  The
{\em classical\/} one introduces $\Bigclas$ as a certain suboperad of
the endomorphism operad $\Endop_{\CHclas*}$ of the graded vector space
$\CHclas*$, and the {\em operadic\/} one introduces $\Big$ as a
suboperad of the endomorphism operad $\Endop_{\CH*}$. Here $A$ is a
{\it generic}, in the sense of Definition~\ref{generic}, 
unital associative algebra. The difference between $\Bigclas$
and $\Big$ is merely conventional; the operad $\Bigclas$ is the
operadic suspension $\ss \Big$ of the operad
$\Big$~\cite[Definition~II.3.15]{markl-shnider-stasheff:book} while,
of~course, $\Endop_{\CHclas*} \cong \ss \Endop_{\CH*}$.
\end{cvo}

\begin{convention}
\label{Jak_to_dopadne_s_Jarkou?}
In sections~\ref{s1} and~\ref{brr} we accept the {\em operadic\/}
convention because we want to rely on the Koszul sign rule. As
explained above, the operads $\Big$ and $\Bigclas$ differ from each
other only by the
regrading and sign factors.
\end{convention}

\begin{the-big}
\label{bIG}
Recall the dg-operad $\Big = \{\Big(n)\}_{n \geq 0}$ of {\em
  all natural multilinear operations\/} on the (operadic) Hochschild
cochain complex~(\ref{cano}) of a generic associative algebra $A$ (see
Definition~\ref{generic}) with
coefficients in itself introduced in~\cite{markl:CZEMJ07} (but notice that
we are using here the operadic degree convention,
see~\ref{Jak_to_dopadne_s_Jarkou?}, while~\cite{markl:CZEMJ07} uses the
classical one).  

Let $A$ be a unital associative algebra. 
A {\em natural operation\/} in the sense of~\cite{markl:CZEMJ07} is a~linear
combination of compositions of the following `elementary' operations:

(a)~The insertion $\circ_i : \CH {k} \otimes \CH {l} \to \CH
{k+ l}$ given, for $k,l \geq -1$ and $0 \leq i \leq k$, by the formula
\[
\circ_i(f,g)(\Rada a0{k+l}) :=  (-1)^{il} f(\Rada a0{i-1},
g(\Rada ai{i+l}),\Rada a{i+l+1}{k+l}), 
\]
for $\Rada a1{k+l-1} \in A$ -- the sign is determined by the Koszul rule!

(b)~Let $\mu : A \ot A \to A$ be the associative product, $\id :A \to
A$ the identity map and $1 \in
A$ the unit. Then elementary operations are also
the `constants' $\mu \in \CH 1$, $\id \in \CH 0$ and  $1 \in
\CH {-1}$.

(c)~The assignment $f \mapsto \sgn(\sigma) \cdot f\sigma$ permuting
the inputs of a cochain $f \in \CH k$ according to a permutation
$\sigma \in \Sigma_{k+1}$ and multiplying by the signature of $\sigma$
is an elementary operation.
\end{the-big}

Let $B(A)^l_{\Rada k1n}$ denote, for $l, \Rada k1n \geq 0$, 
the abelian group of all natural, in the above sense, operations 
\begin{equation}
\label{JarkA}
O:\CH {k_1-1} \otimes \cdots \ot  \CH {k_n-1} \to \CH{ l-1}.
\end{equation} 
The regrading in the above
equation guarantees that the super- and 
subscripts of $B(A)^l_{\Rada k1n}$
will all be non-negative integers. Moreover, with 
this definition the spaces $B(A)^l_{\Rada k1n}$ agree
with the ones introduced in~\cite{batanin-markl}.
The system $B(A)^l_{\Rada k1n}$ clearly forms an $\bbN$-coloured suboperad
$B(A)$ of the endomorphism operad of the $\bbN$-coloured
collection $\{\CH {n-1}\}_{n \geq 0}$.

Recall that the Hochschild differential $\dh : \CH
{n-1} \to \CH {n}$ is, for $n \geq 0$, given by the formula
\begin{eqnarray*}
\dh f(a_0\otimes\ldots\otimes a_n)&:=&
(-1)^{n+1} a_0f(a_1\otimes\ldots\otimes a_n)
+f(a_0\otimes\ldots\otimes a_{n-1})a_n
\\
&& +\sum_{i=0}^{n-1}(-1)^{i+n}f(a_0\otimes\ldots
\otimes a_ia_{i+1}\otimes\ldots\otimes a_n),
\end{eqnarray*}
for $a_i\in A$. Apparently, $\dh$ is  
a natural operation belonging to $B(A)^{n+1}_n$. Therefore, if $O
\in B(A)^{l}_{\Rada k1n}$ is as in~(\ref{JarkA}), one may define $\delta O
\in B(A)^{l+1}_{\Rada k1n}$ and, for $1 \leq i \leq k$, also $\partial_i O \in
B(A)^{l}_{k_1,\ldots,k_{i-1},k_i -1,k_{i+1},\ldots,k_n}$ by
\begin{equation}
\label{popozitri_Praha}
\def\arraystretch{1.4}
\begin{array}{rcl}
\delta O(\Rada f1n) &:=& \dh O(\Rada f1n)\  \mbox { and}
\\
\partial_i O(\Rada f1n) &:=& 
(-1)^{k_i + \cdots+ k_{n} + l + n + i}\cdot
O(\Rada f1{i-1},\dh f_i,\Rada f{i+1}n).
\end{array}
\end{equation}
The sign in the second line of the above display equals 
$(-1)^{\deg(f_1) + \cdots + \deg(f_{i-1})} \cdot (-1)^{\deg(O)}$ as
dictated by the Koszul rule.

It follows from definition that elements of 
$B(A)^l_{\Rada k1n}$ can be represented by linear
combinations of $(l;\Rada k1n)$-trees in the sense of the following
definition in which, as usual, the {\em arity\/} of a vertex of a
rooted tree is the number of its input edges and the {\em legs\/} are
the input edges of a tree, 
see~\cite[II.1.5]{markl-shnider-stasheff:book} for the terminology.

\begin{definition}
\label{d2b}
Let $l,\Rada k1n$ be non-negative integers.  An {\em $(l;\Rada k1n)$-tree\/}
is a planar rooted tree with legs labeled by $\rada 1l$ and three types of
vertices:
\begin{itemize}
\item[(a)] 
`white' vertices of arities $\rada {k_1}{k_n}$ labeled by 
$\rada 1n$, 
\item[(b)] 
`black' vertices of arities $\geq 2$ and
\item[(c)] 
`special' black vertices of arity $0$ (no input edges).
\end{itemize}
We moreover require that there are no edges connecting two
black vertices or a black vertex with a special vertex. For $n=0$ we
allow also the exceptional trees \exeptional\ and \stub\ with no
vertices. 
\end{definition}

We call an internal edge whose initial vertex is special a {\em stub\/} (also
called, in~\cite{kontsevich-soibelman}, a {\em tail\/}). It follows
from definition that the terminal vertex of a stub is white;  the
exceptional tree \stub\ is not a stub. An
example of an $(l;\Rada k1n)$-tree is given in Figure~\ref{fig2}.

\begin{figure}
{
\unitlength=1.2pt
\thicklines
\begin{picture}(120.00,95.00)(-90.00,0.00)
\put(50.00,0.00){\makebox(0.00,0.00){$8$}}
\put(80.00,0.00){\makebox(0.00,0.00){$7$}}
\put(20.00,0.00){\makebox(0.00,0.00){$5$}}
\put(30.00,0.00){\makebox(0.00,0.00){$6$}}
\put(90.00,0.00){\makebox(0.00,0.00){$4$}}
\put(0.00,0.00){\makebox(0.00,0.00){$3$}}
\put(120.00,0.00){\makebox(0.00,0.00){$2$}}
\put(60.00,0.00){\makebox(0.00,0.00){$1$}}
\put(100.00,20.00){\makebox(0.00,0.00){$\bullet$}}
\put(40.00,20.00){\makebox(0.00,0.00){$\bullet$}}
\put(60.00,70.00){\makebox(0.00,0.00){$\bullet$}}
\put(30.00,20.00){\makebox(0.00,0.00){$\bullet$}}
\put(100.00,40.00){\makebox(0.00,0.00){\large$\circ$}}
\put(102.00,42.00){\makebox(0.00,0.00)[lb]{\scriptsize $4$}}
\put(80.00,30.00){\makebox(0.00,0.00){\large$\circ$}}
\put(82.00,32.00){\makebox(0.00,0.00)[lb]{\scriptsize $3$}}
\put(40.00,30.00){\makebox(0.00,0.00){\large$\circ$}}
\put(38.00,32.00){\makebox(0.00,0.00)[rb]{\scriptsize $2$}}
\put(60.00,50.00){\makebox(0.00,0.00){\large$\circ$}}
\put(62.00,52.00){\makebox(0.00,0.00)[lb]{\scriptsize $1$}}
\put(30.00,20.00){\line(0,-1){10.00}}
\put(101.00,38.00){\line(2,-3){19.00}}
\put(100.00,39.00){\line(0,-1){19.00}}
\put(99.00,38.00){\line(-1,-3){9.50}}
\put(60.00,70.00){\line(4,-3){38.00}}
\put(80.00,28.00){\line(0,-1){18.00}}
\put(61.50,48.50){\line(1,-1){17.00}}
\put(60.00,48.00){\line(0,-1){38.00}}
\put(41.00,28.00){\line(1,-2){9.00}}
\put(40.00,28.00){\line(0,-1){8.00}}
\put(38.00,28.00){\line(-1,-1){18.00}}
\put(58.50,48.50){\line(-1,-1){17.00}}
\put(60.00,52.00){\line(0,1){18.00}}
\put(60.00,70.00){\line(-1,-1){60.00}}
\put(60.00,90.00){\line(0,-1){20.00}}
\put(60.00,90.00){\makebox(0.00,0.00)[b]{\boxed{\mbox {\rm \small root}}}}
\end{picture}}
\caption{\label{fig2}
An $(8;3,3,1,3)$-tree representing an operation
in $\Big^8_{3,3,1,3}$. It has 4 white vertices, 2 black
vertices and 2 stubs. We use the convention that directed edges
point upwards so the root is always on the top.}
\end{figure}
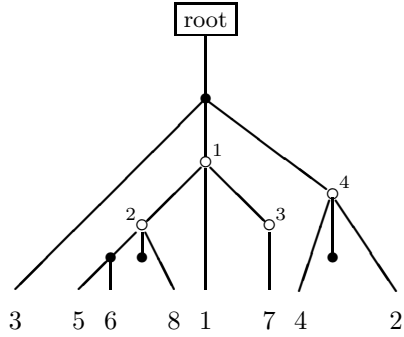

Each $(l;\Rada k1n)$-tree $T$ as in Definition~\ref{d2b} has
its {\em signature\/} $\sigma(T)= \pm 1$ defined as follows. 
Since $T$ is planar, its white vertices are naturally
linearly ordered by walking around the tree counterclockwise,
starting at the root. The first white vertex which one meets is the
first one in this linear order, 
the next white vertex different from the first one is the second in
this linear order, etc. For instance, the labels of the
tree in Figure~\ref{fig2} agree with the ones given by
the natural order, which of course need not always be the case. 
 
One is therefore given a function $w \mapsto p(w)$ that assigns to
each white vertex $w$ of the tree $T$ its position $p(w) \in \{\rada
1n\}$ in the linear order described above. This defines a
permutation $\sigma \in \Sigma_n$ by $\sigma(i) := p(w_i)$, where
$w_i$ is the white vertex labelled by $i$, $1 \leq i \leq n$. Let,
finally, $\sigma(T)$ be the Koszul sign of $\sigma$ permuting $n$
variables $\Rada v1n$ of degrees $\rada{k_1-1}{k_n-1}$, respectively.
In other words, $\sigma(T)$ is determined by
\begin{equation}
\label{koleno}
\sigma(T)\cdot v_1 \land\cdots \land v_n 
= v_{\sigma(1)}\land \cdots\land  v_{\sigma(n)},
\end{equation}
satisfied in the free graded commutative associative
algebra generated by $\Rada v1n$.

An $(l;\Rada k1n)$-tree $T$ determines\label{boli-mne-koleno} 
the natural operation $O_T \in B(A)^l_{\Rada k1n}$
given by decorating, for each $1 \leq i \leq n$, the $i$th white
vertex by $f_i \in \CH {k_i-1}$, the black vertices by the iterated
multiplication, the special vertices by the unit $1 \in A$, and
performing the composition along the tree. The result is then
multiplied by the signature $\sigma(T)$ defined above.

When evaluating on concrete elements, we apply 
the Koszul sign rule and use 
the `desuspended' degrees, that is $f : \otexp An \to A$
is assigned degree $n-1$ and $a \in A$ degree $-1$, 
see~\ref{Jak_to_dopadne_s_Jarkou?}.
For instance, the tree in
Figure~\ref{fig2} represents the operation
\[
O(f_1,f_2,f_3,f_4)(\Rada a18) :=
- a_3 f_1(f_2(a_5a_6,1,a_8),a_1,f_3(a_7))f_4(a_4,1,a_2),
\]
$\Rada a18 \in A$,
where, as usual, we omit the symbol for the iteration of the 
associative multiplication $\mu$. The minus sign in the right hand
side follows from the
Koszul rule explained above. The exceptional $(1;)$-tree \exeptional\
represents the identity $\id \in \CH 0$. 

\begin{notation}
For each $l,\Rada k1n \geq 0$ denote by $B^l_{\Rada k1n}$ the free
abelian group spanned by all $(l,\Rada k1n)$-trees. The
correspondence $T \mapsto O_T$ defines,  for each associative algebra 
$A$, a linear epimorphism $\omega_A : B^l_{\Rada k1n} \epi
B(A)^l_{\Rada k1n}$.
\end{notation}

Let $T'$ be an $(l';\Rada {k'}1n)$-tree, $T''$ an $(l'';\Rada
{k''}1m)$-tree and assume that $l''= k'_i$ for some $1 \leq i \leq n$.
The {\em $i$th vertex insertion\/} assigns to
$T'$ and $T''$ the tree $T' \circ_i T''$ obtained by replacing the
white vertex of $T'$ labelled $i$ by $T''$. It may happen that
this replacement creates edges connecting black vertices. In that
case it is followed by collapsing these edges. The above construction
extends into a linear operation
\[
\circ_i : B^{l'}_{\Rada {k'}1n} \otimes B^{l''}_{\Rada {k''}1m} 
\to B^{l'}_{\Rada {k'}1{i-1},\Rada {k''}1m,\Rada {k'}{i+1}n},\
1 \leq i \leq n,\ l''= k'_i.
\]
Recall the following:

\begin{proposition}[\cite{batanin-markl}]
The spaces $B^l_{\Rada k1n}$ assemble into an 
$\bbN$-coloured operad $B$ with the operadic composition given by the vertex
insertion and the symmetric group relabelling the white vertices. 
With this structure, the maps $\omega_A :\! B^l_{\Rada k1n} \hskip -.3em \epi
B(A)^l_{\Rada k1n}$  form an epimorphism $\omega_A : B \epi
B(A)$ of $\bbN$-coloured operads.
\end{proposition}

In~\cite{batanin-markl} we formulated the following important:

\begin{definition}
\label{generic}
A unital associative algebra $A$ is {\em generic\/} if the
map $\omega_A \!:\! B \epi B(A)$ is an isomorphism.
\end{definition}

In~\cite{batanin-markl} we also proved that generic algebras
exist; the free associative unital algebra $U := {\mathbb T}
(x_1,x_2,x_3,\ldots)$ generated
by countably many generators $x_1,x_2,x_3,\ldots$ is an example.
We may therefore {\em define\/} the operad $B$ alternatively as the operad of
natural operations on the Hochschild cochain complex of a {\em generic
algebra\/}. 

The differentials~(\ref{popozitri_Praha}) clearly translate, for a
generic $A$, to the tree language of the operad $B$ as follows.
The component $\pa_i$, $1 \leq i \leq n$, of the differential $\pa =
\pa_1 + \cdots + \pa_n$ replaces the white vertex of an $(l;\Rada
k1n)$-tree $T$
labelled $i$ with $k_i \geq 1$ inputs by the linear combination
\begin{equation}
\label{prijede_dnes_Jarka?}
\raisebox{-1.5em}{\rule{0cm}{0cm}}
\lwhite 
+ \hskip .5em
\rwhite 
+
(-1)^{k_i+1} \sum_{1 \leq s \leq k_i-1}
\cwhite
\end{equation}
in which the white vertex has $k_i - 1$ inputs and retains the label
$i$. The result is then multiplied by the overall sign in the second
line of~(\ref{popozitri_Praha}).  In the summation
of~(\ref{prijede_dnes_Jarka?}), the black binary vertex is inserted
into the $s$th input of the white vertex. If the $i$th white vertex of
$T$ has no inputs then $\pa_i(T)=0$.

The differential $\delta$ replaces an $(l;\rada{k_1}{k_n})$-tree
symbolized by the triangle \hskip .2em $\triangle$ \hskip .2em with $l$ inputs 
by the linear combination
\[
\raisebox{-1.5em}{\rule{0cm}{0cm}}
\lamptriangle 
+ \hskip .5em \ramptriangle
+
(-1)^{l} \sum_{1 \leq s \leq l}
\camptriangle
\] 
If a replacement above creates an edge connecting
black vertices, it is followed by collapsing
these edges. 

We finally define the arity $n$ piece of the operad of natural operations as
\[
\Big^*(n) := \prod_{l - (k_1 + \cdots + k_n) + n -1= *} {B^l_{\Rada k1n}},
\]
with the degree $+1$ differential $d : \Big^* \to \Big^{*+1}$ 
defined by $d : = (\pa_1 + \cdots + \pa_n) - \delta$. It is evident that
the collection $\Big = \{\Big ^*(n)\}_{n \geq 0}$, 
with the operadic composition inherited from the inclusion $\Big
\subset \Endop_{\CH*}$ for $A$ generic, is a dg-operad.

The structure of the operad $\Big$ is visualized in
Figure~\ref{jarka1}. We emphasize that the degree $m$-piece of $\Big(n)$ is the
direct {\em product\/}, not the direct sum, of elements on the
diagonal $p+q =m -n +1$ in the $(p,q)$-plane. It follows from our
definitions that the Hochschild complex $\CH*$ of an arbitrary unital
associative $A$ is a natural $\Big$-algebra.
\begin{figure}
{
\unitlength=1pt
\begin{picture}(200.00,163.00)(-90.00,40.00)
\put(150.00,52){\makebox(0.00,0.00){$\vdots$}}
\put(100.00,52){\makebox(0.00,0.00){$\vdots$}}
\put(50.00,52){\makebox(0.00,0.00){$\vdots$}}
\put(200.00,80.00){\makebox(0.00,0.00){$\cdots$}}
\put(200.00,110.00){\makebox(0.00,0.00){$\cdots$}}
\put(200.00,140.00){\makebox(0.00,0.00){$\cdots$}}
\put(200.00,170.00){\makebox(0.00,0.00){$\cdots$}}
\put(200.00,200.00){\makebox(0.00,0.00){$\cdots$}}
\put(150.00,120.00){\vector(0,1){10.00}} 
\put(100.00,90.00){\vector(0,1){10.00}}
\put(50.00,90.00){\vector(0,1){10.00}}
\put(150.00,90.00){\vector(0,1){10.00}}
\put(100.00,60.00){\vector(0,1){10.00}}
\put(50.00,60.00){\vector(0,1){10.00}}
\put(150.00,60.00){\vector(0,1){10.00}}
\put(100.00,120.00){\vector(0,1){10.00}}
\put(50.00,120.00){\vector(0,1){10.00}}
\put(150.00,150.00){\vector(0,1){10.00}}
\put(100.00,150.00){\vector(0,1){10.00}}
\put(50.00,150.00){\vector(0,1){10.00}}
\put(150.00,180.00){\vector(0,1){10.00}}
\put(100.00,180.00){\vector(0,1){10.00}}
\put(50.00,180.00){\vector(0,1){10.00}}
\put(165.00,140.00){\vector(1,0){25.00}}
\put(115.00,140.00){\vector(1,0){20.00}}
\put(65.00,140.00){\vector(1,0){20.00}}
\put(10.00,140.00){\vector(1,0){25.00}}
\put(165.00,110.00){\vector(1,0){25.00}}
\put(115.00,110.00){\vector(1,0){20.00}}
\put(65.00,110.00){\vector(1,0){20.00}}
\put(10.00,110.00){\vector(1,0){25.00}}
\put(165.00,80.00){\vector(1,0){25.00}}
\put(115.00,80.00){\vector(1,0){20.00}}
\put(65.00,80.00){\vector(1,0){20.00}}
\put(10.00,80.00){\vector(1,0){25.00}}
\put(10.00,170.00){\vector(1,0){25.00}}
\put(165.00,170.00){\vector(1,0){25.00}}
\put(115.00,170.00){\vector(1,0){20.00}}
\put(65.00,170.00){\vector(1,0){20.00}}
\put(160.00,200.00){\vector(1,0){30.00}}
\put(110.00,200.00){\vector(1,0){30.00}}
\put(60.00,200.00){\vector(1,0){30.00}}
\put(150.00,80.00){\makebox(0.00,0.00){$B(n)^2_3$}}
\put(150.00,110.00){\makebox(0.00,0.00){$B(n)^2_2$}}
\put(150.00,140.00){\makebox(0.00,0.00){$B(n)^2_1$}}
\put(150.00,170.00){\makebox(0.00,0.00){$B(n)^2_0$}}
\put(100.00,80.00){\makebox(0.00,0.00){$B(n)^1_3$}}
\put(100.00,110.00){\makebox(0.00,0.00){$B(n)^1_2$}}
\put(100.00,140.00){\makebox(0.00,0.00){$B(n)^1_1$}}
\put(100.00,170.00){\makebox(0.00,0.00){$B(n)^1_0$}}
\put(50.00,80.00){\makebox(0.00,0.00){$B(n)^0_3$}}
\put(50.00,110.00){\makebox(0.00,0.00){$B(n)^0_2$}}
\put(50.00,140.00){\makebox(0.00,0.00){$B(n)^0_1$}}
\put(50.00,170.00){\makebox(0.00,0.00){$B(n)^0_0$}}
\put(0.00,80.00){\makebox(0.00,0.00){$0$}}
\put(0.00,110.00){\makebox(0.00,0.00){$0$}}
\put(0.00,140.00){\makebox(0.00,0.00){$0$}}
\put(0.00,170.00){\makebox(0.00,0.00){$0$}}
\put(150.00,200.00){\makebox(0.00,0.00){$0$}}
\put(100.00,200.00){\makebox(0.00,0.00){$0$}}
\put(50.00,200.00){\makebox(0.00,0.00){$0$}}
\end{picture}}
\caption{
The structure of the operad $\Big$. In the above
diagram, $B(n)^m_k := \prod_{k_1+\cdots +k_n =k}
B^m_{k_1,\ldots,k_n}$. The vertical arrows are the simplicial
differentials $\pa$ and the horizontal arrows are the cosimplicial
differentials~$\delta$.\label{jarka1}}
\end{figure}

\begin{convention}
\label{syd}
{}From now on, we will assume that $A$ is a {\em generic algebra\/} in
the sense of Definition~\ref{generic} and 
make no distinction between natural operations on the Hochschild
complex of $A$
and the corresponding linear combinations of trees.
\end{convention}

\begin{variant}
\label{varr}
An important suboperad of $\Big$ is the suboperad $\nBr$ generated
by trees {\em without\/} stubs and without \stub. 
The operad $\nBr$ is the operad of
all natural multilinear operations on the Hochschild complex of a~{\em
non-unital\/} generic associative algebra. It is generated by natural
operations (a)--(c) above but without the unit $1 \in \CH {-1}$ in
(b). Let us denote by $\Bn^l_{\Rada k1n}$ the space of all
operations~(\ref{JarkA}) of this restricted type.
An important feature of the operad $\nBr$ is that it is, in a
certain sense, bounded. Indeed, one may easily prove that
$\Bn^l_{\Rada k1n} =0$ if $k_1 + \cdots + k_n -l \geq n$,   
see Figure~\ref{jArka}.

\begin{figure}
{
\unitlength=1pt
\begin{picture}(200.00,203.00)(-85.00,0.00)
\put(150.00,117.50){\makebox(0.00,0.00){$\vdots$}}
\put(100.00,117.50){\makebox(0.00,0.00){$\vdots$}}
\put(50.00,117.50){\makebox(0.00,0.00){$\vdots$}}
\put(200.00,0.00){\makebox(0.00,0.00){$\cdots$}}
\put(200.00,30.00){\makebox(0.00,0.00){$\cdots$}}
\put(200.00,60.00){\makebox(0.00,0.00){$\cdots$}}
\put(200.00,90.00){\makebox(0.00,0.00){$\cdots$}}
\put(200.00,140.00){\makebox(0.00,0.00){$\cdots$}}
\put(200.00,170.00){\makebox(0.00,0.00){$\cdots$}}
\put(200.00,200.00){\makebox(0.00,0.00){$\cdots$}}
\put(150.00,10.00){\vector(0,1){10.00}}
\put(150.00,40.00){\vector(0,1){10.00}}
\put(100.00,40.00){\vector(0,1){10.00}}
\put(150.00,70.00){\vector(0,1){10.00}}
\put(100.00,70.00){\vector(0,1){10.00}}
\put(50.00,70.00){\vector(0,1){10.00}}
\put(150.00,100.00){\vector(0,1){10.00}}
\put(100.00,100.00){\vector(0,1){10.00}}
\put(150.00,120.00){\vector(0,1){10.00}}
\put(100.00,120.00){\vector(0,1){10.00}}
\put(50.00,120.00){\vector(0,1){10.00}}
\put(50.00,100.00){\vector(0,1){10.00}}
\put(150.00,150.00){\vector(0,1){10.00}}
\put(100.00,150.00){\vector(0,1){10.00}}
\put(50.00,150.00){\vector(0,1){10.00}}
\put(150.00,180.00){\vector(0,1){10.00}}
\put(100.00,180.00){\vector(0,1){10.00}}
\put(50.00,180.00){\vector(0,1){10.00}}
\put(160.00,0.00){\vector(1,0){30.00}}
\put(170.00,30.00){\vector(1,0){20.00}}
\put(110.00,30.00){\vector(1,0){20.00}}
\put(165.00,60.00){\vector(1,0){25.00}}
\put(115.00,60.00){\vector(1,0){20.00}}
\put(60.00,60.00){\vector(1,0){25.00}}
\put(170.00,90.00){\vector(1,0){20.00}}
\put(120.00,90.00){\vector(1,0){10.00}}
\put(70.00,90.00){\vector(1,0){10.00}}
\put(10.00,90.00){\vector(1,0){20.00}}
\put(165.00,140.00){\vector(1,0){25.00}}
\put(115.00,140.00){\vector(1,0){20.00}}
\put(65.00,140.00){\vector(1,0){20.00}}
\put(10.00,140.00){\vector(1,0){25.00}}
\put(10.00,170.00){\vector(1,0){25.00}}
\put(165.00,170.00){\vector(1,0){25.00}}
\put(115.00,170.00){\vector(1,0){20.00}}
\put(65.00,170.00){\vector(1,0){20.00}}
\put(160.00,200.00){\vector(1,0){30.00}}
\put(110.00,200.00){\vector(1,0){30.00}}
\put(60.00,200.00){\vector(1,0){30.00}}
\put(150.00,30.00){\makebox(0.00,0.00){$\Bn(n)^2_{n+1}$}}
\put(150.00,60.00){\makebox(0.00,0.00){$\Bn(n)^2_n$}}
\put(150.00,90.00){\makebox(0.00,0.00){$\Bn(n)^2_{n-1}$}}
\put(150.00,140.00){\makebox(0.00,0.00){$\Bn(n)^2_1$}}
\put(150.00,170.00){\makebox(0.00,0.00){$\Bn(n)^2_0$}}
\put(100.00,60.00){\makebox(0.00,0.00){$\Bn(n)^1_n$}}
\put(100.00,90.00){\makebox(0.00,0.00){$\Bn(n)^1_{n-1}$}}
\put(100.00,140.00){\makebox(0.00,0.00){$\Bn(n)^1_1$}}
\put(100.00,170.00){\makebox(0.00,0.00){$\Bn(n)^1_0$}}
\put(50.00,90.00){\makebox(0.00,0.00){$\Bn(n)^0_{n-1}$}}
\put(50.00,140.00){\makebox(0.00,0.00){$\Bn(n)^0_1$}}
\put(50.00,170.00){\makebox(0.00,0.00){$\Bn(n)^0_0$}}
\put(150.00,0.00){\makebox(0.00,0.00){$0$}}
\put(100.00,30.00){\makebox(0.00,0.00){$0$}}
\put(50.00,60.00){\makebox(0.00,0.00){$0$}}
\put(0.00,90.00){\makebox(0.00,0.00){$0$}}
\put(0.00,140.00){\makebox(0.00,0.00){$0$}}
\put(0.00,170.00){\makebox(0.00,0.00){$0$}}
\put(150.00,200.00){\makebox(0.00,0.00){$0$}}
\put(100.00,200.00){\makebox(0.00,0.00){$0$}}
\put(50.00,200.00){\makebox(0.00,0.00){$0$}}
\end{picture}}
\caption{
The structure of the non-unital operad $\nBr$. In the diagram, 
$\Bn(n)^m_k := \prod_{k_1+\cdots +k_n =k}
\Bn^m_{k_1,\ldots,k_n}$.
\label{jArka}}
\end{figure}

One also has the quotient $\NormB$\label{Jary} of the collection $\Big$
modulo the trees with stubs. As explained
in~\cite{batanin-markl}, $\NormB$ forms an operad which is in fact the
componentwise simplicial normalization of $\Big$. 
The operad $\NormB$ acts on the {\em normalized\/} Hochschild
complex of a {\em unital\/} algebra. One has the diagram of
operad maps
\begin{equation}
\label{koleno_porad_boli}
\nBr \stackrel\iota\hookrightarrow \Big \stackrel\pi\twoheadrightarrow \NormB,
\end{equation}
in which the projection $\pi$ is a weak equivalence and the components
$\pi\iota(n)$ of the composition $\pi\iota$ are isomorphisms for each
$n \geq 1$. If $\mathfrak{U}$ denotes the functor that replaces the arity
zero component of a dg-operad by the trivial abelian group, then
$\mathfrak{U}(\pi\iota)$ is a dg-operad isomorphism $\mathfrak{U}(\nBr) \cong
\mathfrak{U}(\NormB)$.
\end{variant}

\begin{TT}\label{TTTT}
There is also a suboperad $\Tam$ of $\Big$ generated by\label{tamm}
elementary operations of types~(a) and~(b) only, without the use of
permutations in~(c). Its arity-$n$ piece equals
\[
\Tam^*(n) :=\prod_{l - (k_1 + \cdots + k_n) + n -1 = *} {T^l_{\Rada k1n}}, 
\]
where operations in $T^l_{\Rada k1n}$ are represented by linear
combinations of {\em unlabeled\/} $(l;\Rada k1n)$-trees, that
is,\label{stryc} planar trees as in Definition~\ref{d2b} but without
the labels of the legs. The inclusion $T^l_{\Rada k1n} \hookrightarrow
B^l_{\Rada k1n}$ is realized by labeling the legs of an unlabeled tree
from the left to the right in the orientation given by the planar
embedding.  The groups $T^l_{\Rada k1n}$ form a coloured operad $T$
and the inclusion above is the inclusion of operads $T \hookrightarrow
B$.

The operad $\Tam$ is the condensation of $T$ and it is a chain version
of the operad considered in~\cite[Section~3]{tamarkin-tsygan:LMP01}.
There is also the\label{tammm-non} operad $\nTam := \nBr \cap \Tam$
generated by unlabeled trees without stubs and without \stub.  It is
clear that $\nTam$ is bounded in the same way as $\nBr$.  We finally
have the normalized Tamarkin-Tsygan operad $\NormT$
defined\label{tammm} as the image of $\Tam$ under the canonical
projection $\pi: \Big \epi \NormB$. One has the diagram $\nTam
\stackrel\iota\hookrightarrow \Tam \stackrel\pi\twoheadrightarrow
\NormT$ with the properties analogous to that
of~(\ref{koleno_porad_boli}).

Summing up, we have the following $\bbN$-coloured operads:

- the operad $B$ whose piece $B^l_{\Rada k1n}$ equals 
the span of the set of all $(l;\Rada k1n)$-trees,

- the operad $\Bn$ whose piece $\Bn^l_{\Rada k1n}$ is
the span of the set of all $(l;\Rada k1n)$-trees without stubs and
without \stub\ if $n=l=0$,

- the operad $T$ whose piece $T^l_{\Rada k1n}$ equals
the span of the set of all unlabeled $(l;\Rada k1n)$-trees, and

- the operad $\Tn = T \cap \Bn$ whose piece
$\Tn^l_{\Rada k1n}$ is the span of the set of all unlabeled $(l;\Rada
k1n)$-trees without stubs and
without \stub\ if $n=l=0$.
\end{TT}

We close this section by recalling the isomorphism between the set of
unlabeled $(l;\Rada k1n)$-trees and $\Lat 2(\Rada k1n;l)$ constructed
in the proof\label{jaruskA} of~\cite[Proposition~2.14]{bb}. Let $T$ be
an unlabeled $(l;\Rada k1n)$-tree. We run around $T$ counterclockwise
via the unique edge-path that begins and ends at the root and goes
through each edge of $T$ exactly twice (in opposite directions).  The
lattice path $\varphi_T : [l+1] \to [k_1 + 1] \ot \cdots \ot [k_n +
1]$ corresponding to $T$ starts at the `lower left' corner with
coordinates $(\rada 00)$ and advances according the following rules:

- when the edge-path hits the white vertex labeled $i$, $1 \leq i \leq
  n$, we advance $\varphi_T$ in the direction of the vector $d_i
  :=(0,\ldots,1,\ldots,0)$ (1 at the $i$th place),

- when the edge-path hits the leg, we do not move but
  increase the marking of our position by one.

The correspondence  $T \mapsto \varphi_T$ is
illustrated in Figure~\ref{JarunkA}. 

\begin{proposition}[\cite{bb}, Proposition~2.14]
\label{pozitri_uvidim_Jarusku} 
The above correspondence induces an isomorphism of coloured operads
$T$ and $\Lat 2,$ and hence, the isomorphism between
\hbox{$\ss\!\Tam$} and $|\Lat 2|$.\footnote{The operadic suspension
$\ss$ applied to $\Tam$ is a consequence of
  Convention~\ref{Jak_to_dopadne_s_Jarkou?}.}
\end{proposition}

\begin{figure}[t]
{
\unitlength=0.6pt
\begin{picture}(100.00,110.00)(-240.00,0.00)
\thicklines
\put(-130,0){
\put(50.00,100.00){\makebox(0.00,0.00)[b]{\boxed{\mbox {\rm \small root}}}}
\put(0.00,50.00){\makebox(0.00,0.00)[b]{$T:$}}
\put(90.00,30.00){\makebox(0.00,0.00){\scriptsize $1$}}
\put(60.00,60.00){\makebox(0.00,0.00){\scriptsize $2$}}
\put(50.00,50.00){\makebox(0.00,0.00){$\circ$}}
\put(80.00,20.00){\makebox(0.00,0.00){$\circ$}}
\put(70.00,30.00){\makebox(0.00,0.00){$\bullet$}}
\put(60.00,0.00){\line(1,3){10.00}}
\put(70.00,0.00){\makebox(0.00,0.00){$\bullet$}}
\put(20.00,20.00){\makebox(0.00,0.00){$\bullet$}}
\put(50.00,80.00){\makebox(0.00,0.00){$\bullet$}}
\put(30.00,60.00){\line(1,1){20.00}}
\put(100.00,0.00){\line(-1,1){17.00}}
\put(69.00,-3.00){\line(1,2){10.00}}
\put(77.00,23.00){\line(-1,1){24.00}}
\put(50.00,0.00){\line(0,1){47.00}}
\put(30.00,0.00){\line(-1,2){10.00}}
\put(0.00,0.00){\line(1,1){47.00}}
\put(50.00,101.00){\line(0,-1){47.00}}
}
\put(140,0){
\unitlength=1.8em
\thinlines
\put(0,0){\line(1,0){3}}
\put(0,1){\line(1,0){3}}
\put(0,2){\line(1,0){3}}
\put(0,3){\line(1,0){3}}
\put(0,4){\line(1,0){3}}
\put(0,0){\line(0,1){4}}
\put(1,0){\line(0,1){4}}
\put(2,0){\line(0,1){4}}
\put(3,0){\line(0,1){4}}
\thicklines
\put(-1,1.5){\put(.1,.1){\makebox(0.00,0.00)[rb]{$\longmapsto \hskip
      2em \varphi_T:$}}}
\put(0,0){\makebox(0.00,0.00){$\bullet$}}
\put(0,0){\put(.1,.1){\makebox(0.00,0.00)[lb]{\scriptsize $1$}}}
\put(0,0){\vector(0,1){1}}
\put(0,1){\makebox(0.00,0.00){$\bullet$}}
\put(0,1){\put(.1,.1){\makebox(0.00,0.00)[lb]{\scriptsize $2$}}}
\put(0,1){\vector(0,1){1}}
\put(0,2){\makebox(0.00,0.00){$\bullet$}}
\put(0,2){\put(.1,.1){\makebox(0.00,0.00)[lb]{\scriptsize $1$}}}
\put(0,2){\vector(0,1){1}}
\put(0,3){\makebox(0.00,0.00){$\bullet$}}
\put(0,3){\put(.1,.1){\makebox(0.00,0.00)[lb]{\scriptsize $1$}}}
\put(0,3){\line(1,0){1}}
\put(1,3){\makebox(0.00,0.00){$\bullet$}}
\put(1,3){\put(.1,.1){\makebox(0.00,0.00)[lb]{\scriptsize $0$}}}
\put(1,3){\vector(1,0){1}}
\put(2,3){\makebox(0.00,0.00){$\bullet$}}
\put(2,3){\put(.1,.1){\makebox(0.00,0.00)[lb]{\scriptsize $1$}}}
\put(2,3){\line(1,0){1}}
\put(3,3){\makebox(0.00,0.00){$\bullet$}}
\put(3,3){\put(.1,.1){\makebox(0.00,0.00)[lb]{\scriptsize $0$}}}
\put(3,3){\vector(0,1){1}}
\put(3,4){\put(.1,.1){\makebox(0.00,0.00)[lb]{\scriptsize $0$}}}
}
\end{picture}}
\caption{\label{JarunkA}
An unlabeled $(6;2,2)$-tree $T$ and the corresponding lattice path
$\varphi_T \in \Lat c(2,2;6)$.}
\end{figure}
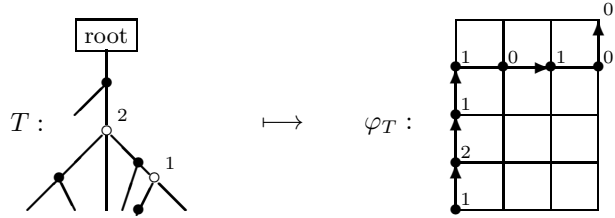

More conceptually, the difference between the $\bbN$-coloured operads 
$T$ and $B$ and the corresponding operads $\Tam$ and $\Big$ can be
explained as follows.  Let $O$ and $O_1$ be the categories of operads
and of nonsymmetric operads in the category of chain complexes
$\Chain$ correspondingly. There is the forgetful functor ${\it
  Des}_1: O\rightarrow O_1$ which forget the symmetric group
actions. Let $\Mon$ be the {\em nonsymmetric} operad for unital monoids.

\begin{definition} 
The category of {\it multiplicative nonsymmetric operads} is the
comma-category $\Mon/O_1$, see~\cite{gerstenhaber-voronov:IMRN95}. The category of {\it
  multiplicative operads} is the comma-category $\Mon/{\it Des}_1$.
\end{definition}
So, a multiplicative operad is an operad $A$ equipped with a structure
morphism $p:\Mon \rightarrow {\it Des}_1(A)$. Equivalently, by adjunction, a
structure morphism can be replaced by a morphism $\UAss \rightarrow A,$
where $\UAss$ is the operad for unital associative algebras.

The description in \cite[1.5.6]{berger-moerdijk:CM07} of the coloured operad
whose algebras are symmetric operads, readily implies the following
proposition which illuminates the main result of \cite{batanin-markl}.

\begin{proposition} 
The category of algebras over the coloured operad $T$ is isomorphic to
the category of multiplicative nonsymmetric operads.  The category of
algebras of the coloured operad $B$ is isomorphic to the category of
multiplicative operads.  Under this identification, the inclusion
$T\hookrightarrow B$ induces the forgetful functor from multiplicative
operads to nonsymmetric multiplicative operads.
\end{proposition}

\section{Operads of braces}
\label{brr}

Throughout this section we use Convention~\ref{syd}.
There is another very important suboperad $\Br$ of $\Big$ generated by braces,
cup-products and the unit whose normalized version was introduced
in~\cite[Section~1]{mcclure-smith} under the notation $\mathcal H$. 
Let us recall its definition.
The operad $\Br$ is the suboperad of the operad 
$\Big$ generated by the following operations.

(a) The {\em cup product\/} $- \cup - :\CH * \ot \CH * \to \CH *$
defined by $f
\cup g := \mu(f,g)$.

(b) The constant $1 \in \CH {-1}$.

(c) The {\em braces\/} $-\{\rada --\} : 
\otexp {\CH *}n \to \CH*$, $n \geq 2$, given by
\begin{equation}
\label{braaace}
f\{\Rada g2n\} := \sum f(\rada {\id}{\id},g_2,\rada
{\id}{\id},g_n,\rada {\id}{\id}),
\end{equation}
where $\id$ is the identity map of $A$ and the summation runs over all
possible substitutions of $\Rada g2n$ (in that order) into
$f$. 

Notice that, for $f \in \CH k$ and $g \in \CH l$,  
the cup product $f \cup g \in \CH
{k+l+1}$ evaluated at $\Rada a0{k+l+1} \in A$ equals
\begin{equation}
\label{cuuup}
(f \cup g)(\Rada a0{k+l+1})= 
(-1)^{(k+1)l}
f(\Rada a0k)g(\Rada a{k+1}{k+l+1}),
\end{equation}
with the sign dictated by the Koszul rule. 
This formula differs from the original
one~\cite[Section~7]{gerstenhaber:AM63} due to a different degree
convention used here, see~\ref{Jak_to_dopadne_s_Jarkou?}. 
We leave as an exercise to write a
similar explicit formula for the brace.

The brace operad has also its non-unital version\label{brr-non} $\nuBr
:= \nBr \cap \Br$ generated by elementary operations (a) and (c). One
can verify\label{brrr} that both $\Br$ and $\nuBr$ are indeed
dg-suboperads of $\Big$, see~\cite{mcclure-smith}.  We also denote by
$\Norm(\Br) \subset \Norm(\Big)$ the image of $\Br$ under the\label{brr-norm}
projection $\Big \epi \NormB$. One has again an analog $\nuBr
\stackrel\iota\hookrightarrow \Br \stackrel\pi\twoheadrightarrow
\Norm(\Br)$ of~(\ref{koleno_porad_boli}).

Let us describe the
operad $\Br$, its suboperad $\nuBr$ and its quotient $\NormB$ in
terms of trees.

\begin{definition}
\label{Def11}
Let $\Rada k1n$ be integers. An {\em amputated $(\Rada
k1n)$-tree\/} is an $(0;\Rada k1n)$-tree in the sense of
Definition~\ref{d2b}. We denote by $A_{\Rada k1n}$ the (finite) set of
all amputated $(\Rada k1n)$-trees,  by  $\Norm(A)_{\Rada k1n}$ its subset
consisting of amputated $(\Rada k1n)$-trees without stubs and
$\nA_{\Rada k1n}$ the set that equals  $\Norm(A)_{\Rada k1n}$ for $n
\geq 1$ and is $\emptyset$ for $n=0$. 
\end{definition}

\begin{proposition}
\label{JaRkA}
For each $n \geq 0$ and $d \leq n-1$, there is a natural isomorphism
\[
w: \Span(\{A_{\Rada k1n};\ n-1 -(k_1 + \cdots +k_n) = d\})  \cong \Br^d(n)
\]
which restricts to the isomorphism (denoted by the same symbol)
\[
w: \Span(\{\nA_{\Rada k1n};\ n-1 -(k_1 + \cdots +k_n) = d\})  \cong
\NBrr d(n).
\]
and projects into the isomorphism (denoted again by the same symbol)
\[
w: \Span(\{\Norm(A)_{\Rada k1n};\ n-1 -(k_1 + \cdots +k_n) = d\})  \cong
\Norm(\Br)^d(n).
\]
\end{proposition}

The map $w$ is defined in
formula~(\ref{Syd}) below.   
{}From the reasons apparent later we call it the
{\em whiskering\/}. The proof of the proposition is postponed to
page~\pageref{Jarca-pusa}. Before we give the definition of $w$, 
we illustrate the
notion of amputated trees in the following: 

\begin{example}
The space $\Br^*(0)$ is concentrated in degree $-1$, 
\[
\Br^*(0) = \Br^{-1}(0) = \Span(\AAblack),
\] 

\noindent 
while $\NBrr*(0) = 0 =\Span(\emptyset)$.
The space $\Br^{-d}(1)$ is, for $d \geq 0$, the span of the single element
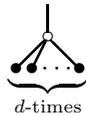
\[
\AAd
\] 
while $\NBrr*(1) = \NBrr0(1) = {\Span (\AAnulamin)}$.
Similarly
\[
\Br^1(2) = \NBrr1(2) = 
\Span\left(\rule{0pt}{1.5em}\right.\AAparm12,\AAparm21
\left.\rule{0pt}{1.5em}\right),\
\NBrr{0}(2) = \Span\left(\rule{0pt}{1.5em}\right.\VVparm12,\VVparm21
\left.\rule{0pt}{1.5em}\hskip -.3em \right)
\]
and
\[
\Br^{0}(2) = \NBrr{0}(2) \oplus
\Span\left(\rule{0pt}{1.5em}\right.\AAlparm12,\AAlparm21,\AArparm12,\AArparm21
\left.\rule{0pt}{1.5em}\right).
\]
\end{example}

\begin{definition}
We call an unlabeled $(l;\Rada k1n)$-tree {\em amputable\/} if all
terminal vertices of its legs are white. For such a tree $T$ we denote
by $\amp(T)$ the amputated $(\Rada k1n)$-tree obtained from $T$ by
removing all its legs.
\end{definition}

\begin{example}
The $(1;1,1)$-tree \ $\amputabletree 12$ \ is amputable, and 
\[
\rule{0pt}{3em}
\amp \left(\rule{0pt}{1.5em}\right.\amputabletree 12
\left.\rule{0pt}{1.5em}\right) = \amputatedtree 12.
\]
The $(2;1)$-tree  $\notamputabletree$ is not amputable.
\end{example}

For each amputated $(\Rada k1n)$-tree $S$ we define the
whiskering to be the product
\begin{equation}
\label{Syd}
\whis(S) := \prod(T;\ \mbox{$T$ is an amputable tree such that $\amp(T) = S$}).
\end{equation}
Recall that, by Convention~\ref{syd}, we interpret
the unlabeled trees in the right hand side as 
operations in $\Tam^d(n)\subset \Big^d(n)$, $d = n-1 -(k_1 +\cdots+ k_n)$, via 
the correspondence $T \leftrightarrow O_T$ introduced on 
page~\pageref{boli-mne-koleno}. An equivalent definition in terms of the
whiskered insertion into a corolla is given in~(\ref{smutek_od_rana}).

\begin{example}
\label{Hambos}
Of course, $w(\stub\ ) =\stub$ \hskip .5em represents the unit $1 \in \CH {-1}$.
The element given by the whiskering of $\AAnulamin$,
\[
w(\AAnulamin) = \prod_{d \geq 0} \AAmodifiedd \in \nTam \subset \nBr(1),
\]
is the identity  $f \mapsto f$, i.e., the unit of the operad $\Big$.
The whiskering of $\AAparm12$,
\[
w(\AAparm12) = \AAparm12 \sqcup \AArparmm 12 \sqcup  \AAlparmm 12
\sqcup  \AArrparmm 12 \sqcup  
\AAlrparm 12 \sqcup  \AAllparmm 12   \sqcup 
\cdots,
\]
gives the cup product~(\ref{cuuup}).  
The whiskering of the element
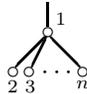
\[
\Abracemod n
\] 
gives the brace~(\ref{braaace}). In particular,
\raisebox{.2em}{$\VVparm12$} gives
Gerstenhaber's $\circ$-product and 
\raisebox{.2em}{$\VVparm12 -\  \VVparm 21$} the
Gerstenhaber bracket. 
Observe that the whiskering of the tree
\begin{equation}
\label{zas_v_Bonnu}
\AAparm21
\end{equation}
is the operation that assigns to $f \in \CH m$ and $g \in \CH n$ the
expression $(-1)^{mn} g \cup f$. The sign comes from the tree
signature factor~(\ref{koleno}) in the definition of the operation
$O_T$, because the order of the white vertices of the
tree~(\ref{zas_v_Bonnu}) and its whiskerings does not agree with the
natural planar one.
\end{example}

We are going to define operations $\pa$ and $\circ_i$ acting on
amputated trees that translate, via the whiskering~(\ref{Syd}), into
the dg-operad structure of $\Br$.  For an amputated $(\Rada k1n)$-tree
$S$ as in Definition~\ref{Def11} denote $\pa(S) := \pa_1(S) + \cdots +
\pa_n(S)$, where $\pa_i(S)$ is, for $k_i \geq 1$, the linear
combination of amputated trees obtained by replacing the $i$th white
vertex of $S$ by~(\ref{prijede_dnes_Jarka?}) followed by the
contraction of edges connecting black vertices if
necessary.\label{Jaruska-Beruska} For $k_i=0$ we put $\pa_i(S) =0$.

The description of the $\circ_i$-operations is more delicate.
Following~\cite[Section~5.2]{kontsevich-soibelman}, define the {\em set of
angles\/} of an amputated $(\Rada k1n)$-tree $S$ to be the disjoint
union
\[
\Angl(S) := \bigsqcup_{1 \leq i \leq n} \{0,\ldots,k_i\}.
\]
Angles come with a natural linear order whose definition is
clear from Figure~\ref{prazdnota} borrowed from~\cite{kontsevich-soibelman}.
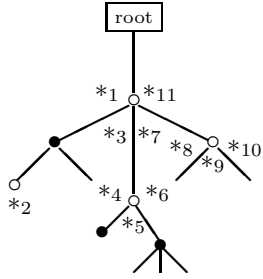
\begin{figure}
\thicklines
{
\unitlength=1pt
\begin{picture}(40.00,85.00)(-185.00,-65.00)
\put(0.00,3.00){\line(0,1){23.00}}
\put(-2,-2){\line(-2,-1){26.00}}
\put(2,-2){\line(2,-1){26.00}}
\put(32.00,-18.00){\line(1,-1){12.00}}
\put(28.00,-18.00){\line(-1,-1){12.00}}
\put(-28.00,-18.00){\line(1,-1){12.00}}
\put(-32.00,-18.00){\line(-1,-1){12.00}}
\put(0.00,-2.00){\line(0,-1){33.00}}
\put(-2.00,-40.00){\line(-1,-1){10.00}}
\put(1.00,-40.00){\line(2,-3){10.00}}
\put(10.00,-55.00){\line(0,-1){10.00}}
\put(10.00,-55.00){\line(1,-1){10.00}}
\put(10.00,-55.00){\line(-1,-1){10.00}}
\put(10.00,-55.00){\makebox(0.00,0.00){\large$\bullet$}}
\put(-12.00,-50.00){\makebox(0.00,0.00){\large$\bullet$}}
\put(0.00,-38.00){\makebox(0.00,0.00){\large$\circ$}}
\put(-45.00,-32.00){\makebox(0.00,0.00){\large$\circ$}}
\put(30.00,-16.00){\makebox(0.00,0.00){\large$\circ$}}
\put(-30.00,-16.00){\makebox(0.00,0.00){\large$\bullet$}}
\put(0.00,0.00){\makebox(0.00,0.00){\large$\circ$}}

\put(-10.00,3.00){\makebox(0.00,0.00){$*_1$}}
\put(-43.00,-40.00){\makebox(0.00,0.00){$*_2$}}
\put(-7.00,-13.00){\makebox(0.00,0.00){$*_{3}$}}
\put(-9.00,-35.00){\makebox(0.00,0.00){$*_4$}}
\put(0.00,-47.00){\makebox(0.00,0.00){$*_5$}}
\put(9.00,-35.00){\makebox(0.00,0.00){$*_6$}}
\put(6.00,-13.00){\makebox(0.00,0.00){$*_7$}}
\put(18.00,-18.00){\makebox(0.00,0.00){$*_8$}}
\put(30.00,-24.00){\makebox(0.00,0.00){$*_9$}}
\put(42.00,-18.00){\makebox(0.00,0.00){$*_{10}$}}
\put(10.00,3.00){\makebox(0.00,0.00){$*_{11}$}}
\put(0.00,26.00){\makebox(0.00,0.00)[b]{\boxed{\mbox {\scriptsize root}}}}
\end{picture}}
\caption{\label{prazdnota}
Angles of a tree symbolized by $\Rada *1{11}$. Their linear order,
indicated by the subscripts, is given by
walking around the tree counterclockwise, starting at the
root. Unlike~\cite[Section~5.2]{kontsevich-soibelman}, 
black vertices do not have angles. The labels of white vertices are
not shown.}
\end{figure}
Now, for an amputated $(\Rada {k'}1n)$-tree $S'$, an amputated $(\Rada
{k''}1m)$-tree $S''$ and $1 \leq i \leq n$, define $S' \circ_i S''$ to
be the linear combination
\begin{equation}
\label{Za_tyden_snad_s_Jarkou}
S' \circ_i S'' := \sum_\beta (S' \circ_i S'')_\beta,
\end{equation}
where the sum runs over all (non-strictly) monotonic maps $\beta :
\In(w'_i) \to \Angl(S'')$ from the set of incoming edges of the vertex
$w_i'$ of $S'$ labelled $i$, to the set
of angles of $S''$. 
In the sum, $(S' \circ_i S'')_\beta$ is the tree obtained by removing
the vertex $w'_i$ from $S'$ and replacing it by $S''$, with the
incoming edges of $w_i'$  glued into the angles of $S''$
following $\beta$. An important particular case is $k'_i = 0$
when $w'_i$ has no input edges. Then $S'\circ_i S''$ is
defined as the tree obtained by amputating $w'_i$ from $S'$ and
grafting the root of $S''$ at the place of $w'_i$.

We call the operation $\circ_i$ the {\em whiskered insertion\/}.  A
similar operation defines in~\cite{chapoton-livernet:pre-lie} the
structure of the operad for pre-Lie algebras.  As observed
in~\cite{kontsevich-soibelman}, the whiskering of
Proposition~\ref{JaRkA} can also be expressed as the product
\begin{equation}
\label{smutek_od_rana}
\raisebox{-1.8em}{\rule{0pt}{1pt}}
w(S) = \prod_{d \geq 0} \left(\AAmodifieddd\right) \circ_1  S.
\end{equation}
The following proposition can be verified directly.

\begin{proposition}
\label{whisk}
With $\pa$ and $\circ_i$ as defined above, 
the whiskering of Proposition~\ref{JaRkA} satisfies
\[
w(\pa S) = d (w (S)) \mbox { and }
w(S' \circ_i S'') = w(S') \circ_i w(S''),
\]
for all amputated trees $S$, $S'$, $S''$ and for all $i$ for which the
second equation makes sense. 
\end{proposition}

\begin{example}
We show how the classical calculations
of~\cite{gerstenhaber:AM63} can be concisely performed in
the language of amputated trees (but recall that we are using a
different sign and degree convention, see~\ref{Jak_to_dopadne_s_Jarkou?}). 
Let us start by calculating the
differentials of trees representing the cap product, the circle
product and the Gerstenhaber bracket. 
By definition, one has 
\begin{equation}
\label{H1}
\pa\lstretch1 \hskip -.2em \AAparm12  \hskip -.2em \rstretch1 = 0.
\end{equation}
Since~(\ref{prijede_dnes_Jarka?}) replaces $\VVV$ by $\AAAr + \AAAl$, one gets
\begin{equation}
\label{H2}
\pa\lstretch1 \hskip .2em \VVparm12 \rstretch1 = \AAparm12 + \AAparm21
\end{equation}
which implies that
\begin{equation}
\label{H3}
\pa \lstretch1 \hskip .2em \VVparm12 -  \VVparm 21  \rstretch1 = 
 \AAparm12 + \AAparm21  \hskip .2em - \hskip .2em \AAparm21 - \AAparm21 = 0.
\end{equation}

We want to interpret these equations in terms of operations. To save
the space, let us agree that in the rest of this example $f$ will be an
element of $\CH m$, $g$ an element of $\CH n$ and $h$ an element 
of $\CH k$, $m,n,k \geq -1$ arbitrary.
By Proposition~\ref{whisk},~(\ref{H1}) means that the differential of
the cup product
$\cupop$ recalled in~(\ref{cuuup}) and
considered as an element of $\Big(2)$ is zero,
$d(\cupop) =0$, which, by the definition~(\ref{popozitri_Praha}) 
of the differential in $\Big$ means that 
\[
- \dh(f \cup g) = \dh f \cup g + (-1)^{m} f \cup \dh g.
\]
We recognize~\cite[Eqn.~(20)]{gerstenhaber:AM63} saying that $\cupop $
is a chain operation. Since \raisebox{.2em}{$\VVparm12$} represents the
$\circ$-product,~(\ref{H2}) means that
\[
 f \cup g +  (-1)^{mn} g \cup f= 
\dh f \circ g + (-1)^{m} f \circ\dh g - \dh (f \circ g),
\]
which is the graded commutativity\footnote{Since we use the convention
  in which the cup product has degree $+1$, its commutativity is the
  {\em anti\/}symmetry.} of the cup product up to the homotopy
$\circop$ proved in~\cite[Theorem~3]{gerstenhaber:AM63}.  The origin
of the sign factor at the second term in the right hand side is
explained in Example~\ref{Hambos}.  The meaning of~(\ref{H3}) is that
\[
\dh [f,g] = [\dh f,g] + (-1)^m  [f,\dh g],
\]
so the bracket $[-,-]$ is a chain operation.

Let us investigate the compatibility between the cup product and the
bracket. Since, in $\Big(3)$,  $[-\cup-,-] = [-,-] \circ_1 (-\cup-)$, 
the description of the $\circ_i$-operations
in terms of amputated trees gives that $[f \cup g,h]$ is represented by  
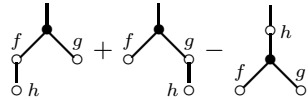
\[
\raisebox{-1.5em}{\rule{0pt}{0pt}}
\AAAlparm fgh + \AAArparm gfh - \AAAparm fgh
\]
where we, for ease of reading, replaced the labels of white vertices
by the corresponding cochains.
Similarly, since $-\cup [-,-] = (-\cup-)\circ_2 [-,-]$ in
$\Big(3)$, $f  \cup [g,h]$ is represented by
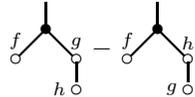
\[
\raisebox{-1.5em}{\rule{0pt}{0pt}}
\AAArparm gfh -  \AAArparm hfg
\]
and, by the same reason, $[f,h] \cup  g$ is represented by 
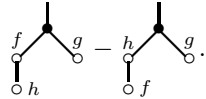
\[
\raisebox{-1.5em}{\rule{0pt}{0pt}}
\AAAlparm fgh -  \AAAlparm hgf.
\]
Combining the above, one concludes that the expression $[f \cup
  g,h] - f  \cup [g,h] - [f,h] \cup  g$ is represented by
\begin{equation}
\label{vecer}
\raisebox{-1.5em}{\rule{0pt}{0pt}}
\AAArparm hfg + \AAAlparm hgf -  \AAAparm fgh.
\end{equation}
Because, by~(\ref{prijede_dnes_Jarka?}), $\pa$ replaces $\trioparm h$
by
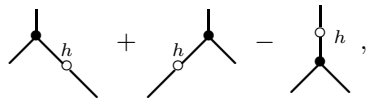
\[
\AAArextparm h  \hskip -.2em + \AAAlextparm h - \trioextparam h,
\raisebox{-2em}{\rule{0pt}{0pt}}
\]
the expression in~(\ref{vecer}) equals
\[
\raisebox{-1.5em}{\rule{0pt}{0pt}}
\pa \lstretch {1.4}  \AAAAparm fgh \rstretch{1.4} .
\]
The meaning of the above calculations is that the bracket and the
cup product are compatible up to the homotopy given by the brace $-\{-,-\}$.
\end{example}

\begin{proof}[Proof of Proposition~\ref{JaRkA}]
\label{Jarca-pusa}
It follows from Proposition~\ref{whisk} that the image of $\whis$ 
contains $\Br$. Indeed,  $\Im(\whis)$ is a suboperad of $\Big$ which, by
Example~\ref{Hambos}, contains the generators of  $\Br$, i.e.~the cup product,
braces and $1$. The map $w$ is clearly a monomorphism, 
since each amputated $(\Rada k1n)$-tree 
$S$ equals the amputated part (i.e.~the component belonging to $\prod
B^0_{\Rada k1n}$) of its whiskering $w(S)$.

Therefore it remains to prove that  $\Im(\whis) \subset \Br$ or,
more specifically, that $w(S) \in \Br(n)$ for each amputated 
$(\Rada k1n)$-tree $S$ and $n \geq 0$.
We need to show that each such $S$ is build up, by the iterated whiskered
insertions $\circ_i$ of~(\ref{Za_tyden_snad_s_Jarkou}) and relabelings
of white vertices, from the `atoms' 
\begin{equation}
\label{brd}
\raisebox{-1.3em}{\rule{0pt}{0pt}}
\AAnulamin,\ \AAblack,\ \cup := \AAparm12 \ 
\mbox { and } \br d := \Abracemod {d\!\!+\!\!1},\ d \geq 1,
\end{equation}
representing the generators of $\Br$. Since the whiskering $w$ is an operad
homomorphism and the atoms are mapped to $\Br$, 
this would indeed imply that $\Im(\whis) \subset \Br$.

The first step is to get rid of the stubs. If $S$ has $s \geq 1$
stubs, we denote by $\overline S$ the tree $S$ with each 
stub replaced by $\AAnulamin$. Let us
label these new white vertices of $\overline S$ by $\rada {n+1}{n+s}$. Then
clearly
\[
S = \pm ( \cdots ((\overline S \circ_{n+1} \AAblack) \circ_{n+2} 
\AAblack) \cdots) \circ_{n+s} \AAblack.
\]
The sign in the above expression, not important for our
purposes, is a consequence of the Koszul sign rule, since
$\AAblack$ represents $1 \in A$ placed in degree $-1$.
So we may suppose that $S$ has no stubs and proceed by induction on the
number of internal edges. Assume that
$S$ has $e$ internal edges. If $e\leq 1$ then $S$ is either $\AAnulamin$
or $\br 1$, so we may assume that $e \geq 2$. We distinguish two cases.

{\em Case 1.\/} The root vertex (i.e.~the vertex adjacent to the root
edge) is white; assume it has $d \geq 1$ input edges. The tree $S$
looks as:
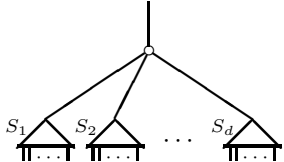
\[
\Jarunka \circ
\]
where $\Rada S1d$ are suitable amputated trees. It is then clear that
$S$ can be obtained from
\[
(\cdots ((\br d \circ_1 S_1) \circ_{2}  S_2) \cdots)
\circ_d S_d,
\]
where $\br d$ is the tree in~(\ref{brd}), by relabeling the white
vertices and changing the sign if necessary. Clearly, each $\Rada S1d$ has
less than $e$ internal edges, and the induction goes on. 

{\em Case 2.\/} The root vertex is black, with $d \geq 2$  inputs. If
$d=2$, we argue as in Case~1, only using $\cup$ instead of
$\br 2$. If $d \geq 3$, we use the equality
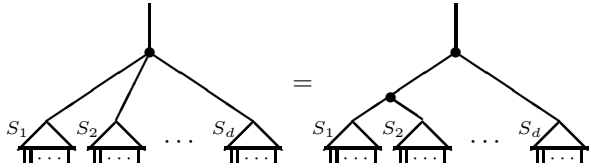
\[
\Jarunka \bullet \hskip 2em  \raisebox{2em}{=}  \hskip 2em \Jarunkamod
\] 
and argue as if $d =2$. This finishes the proof. 
\end{proof}

\begin{figure}
\[
{
\thicklines
\unitlength=.7pt
\begin{picture}(360.00,263.00)(-70.00,0.00)
\put(200.00,100.00){\makebox(0.00,0.00){$\HB$}}
\put(160.00,0.00){\makebox(0.00,0.00){$\nTam$}}
\put(0.00,50.00){\makebox(0.00,0.00){$\HBr$}}
\put(160.00,80.00){\makebox(0.00,0.00){$\Tam$}}
\put(200.00,180.00){\makebox(0.00,0.00){$\Big$}}
\put(0.00,130.00){\makebox(0.00,0.00){$\Br$}}
\put(0,83){
\put(160.00,80.00){\makebox(0.00,0.00){$\NormT$}}
\put(200.00,177.00){\makebox(0.00,0.00){$\NormB$}}
\put(0.00,130.00){\makebox(0.00,0.00){$\Norm(\Br)$}}
}
\put(170.00,92.00){\vector(4,1){20.00}}
\put(24.00,55.00){\line(4,1){100.00}}
\put(0,83){
\put(170.00,92.00){\vector(4,1){20.00}}
\put(24.00,55.00){\line(4,1){100.00}}
}
\put(24.00,124.00){\vector(3,-1){120.00}}
\put(0,83){
\put(30.00,122.00){\vector(3,-1){97.00}}
}
\put(24.00,44.00){\vector(3,-1){120.00}}
\put(32.00,223.00){\vector(4,1){135.00}}
\put(0.00,68.00){\vector(0,1){49.00}}
\put(-5.00,63.00){
\qbezier(0,10)(0,0)(2.5,0)\qbezier(2.5,0)(5,0)(5,10)
}
\put(160.00,18.00){\vector(0,1){49.00}}
\put(155,13.00){
\qbezier(0,10)(0,0)(2.5,0)\qbezier(2.5,0)(5,0)(5,10)
}
\put(200.00,122.00){\vector(0,1){45.00}}
\put(195,114.00){
\qbezier(0,10)(0,0)(2.5,0)\qbezier(2.5,0)(5,0)(5,10)
}
\put(0,83){
\put(0.00,63.00){\vector(0,1){54.00}}
\put(0.00,63.00){\vector(0,1){44.00}}
\put(160.00,13.00){\vector(0,1){54.00}}
\put(160.00,13.00){\vector(0,1){44.00}}
\put(200.00,113.00){\vector(0,1){54.00}}
\put(200.00,113.00){\vector(0,1){44.00}}
}
\put(165.00,103.00){\vector(1,2){33.00}}
\put(166.00,15.00){\vector(1,2){33.00}}
\put(0,83){
\put(161.00,93.00){\vector(1,2){35.00}}
}
\put(-5,170){\makebox(0.00,0.00)[r]{$\pi$}}
\put(206,220){\makebox(0.00,0.00)[l]{$\pi$}}
\put(155,120){\makebox(0.00,0.00)[r]{$\pi$}}

\put(0,-80)
{
\put(-5,170){\makebox(0.00,0.00)[r]{$\iota$}}
\put(206,220){\makebox(0.00,0.00)[l]{$\iota$}}
\put(155,120){\makebox(0.00,0.00)[r]{$\iota$}}
}

\end{picture}}
\]
\caption{Operads of natural operations and their maps; see also the
  glossary on page~\pageref{Smutek_z_Jarky.}. The horizontal
maps are inclusions.\label{fig3}}
\end{figure}
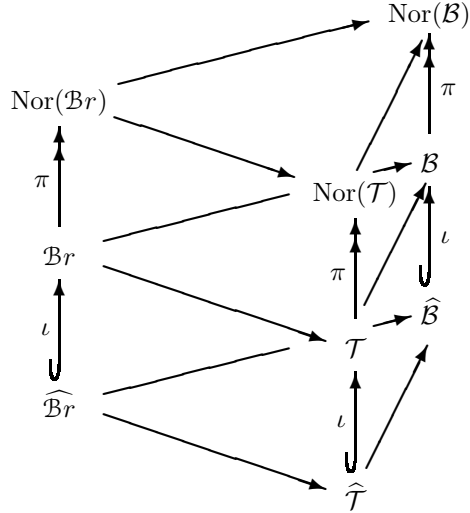

We finish this section by completing the proof of the following
theorem of~\cite{batanin-markl}.

\begin{theorem}
\label{Jaruska1} 
The operads introduced above can be organized into the diagram
in Figure~\ref{fig3}.  In this diagram:
\begin{enumerate}
\item  
Operads in the two upper triangles have the chain homotopy type of the operad
$\SC_{-*}(\Dis)$ of singular chains on the little disks operad $\Dis$
with the inverted grading.
In particular, the big operad $\Big$ of all natural operations has the
homotopy type of $\SC_{-*}(\Dis)$,

\item  
all morphisms between vertices of the two upper triangles 
are weak equivalences,
  
\item operads in the bottom triangle of Figure~\ref{fig3} have the chain
homotopy type of the operad $\SC_{-*}(\Dis)$ with the component of
arity $0$ replaced by the trivial abelian group, and 

\item  all 
morphisms in Figure~\ref{fig3} become weak equivalences after the 
application of the functor~$\mathfrak{U}$ that replaces the component of
arity $0$ of a dg-operad by the trivial abelian group. 
\end{enumerate}
\end{theorem}

\begin{proof}
The only piece of information that was missing in~\cite{batanin-markl}
and for which we had to refer to this paper was that the whiskering $w
: \Br \to \Tam$ is a weak equivalence. This fact follows from
Theorem~\ref{jArKa}, the identification $\ss\!\Tam \cong |\Lat 2|$
established in Proposition~\ref{pozitri_uvidim_Jarusku}, and the
induced identification $\ss\!\Br \cong \Brac 2$ of suboperads.
\end{proof}

\begin{remark} 
Theorem \ref{Jaruska1} shows that, up to homotopy, there is no
difference between actions on the Hochschild cochains of the operads
$\Big$, $\Tam$ and $\Br$, resp.~$\nBr$, $\nTam$ and $\nuBr$ in the nonunital
case, resp.~$\Norm(\Big)$, $\Norm(\Tam)$ and $\Norm(\Br)$ in the
normalized case.
\end{remark}

\appendix
\section{Substitudes, convolution and condensation}
\label{s0}

In this appendix we briefly remind the reader of some categorical
definitions and constructions we use in the paper. Most of the
material is contained in
\cite{DS1},\cite{DS2},\cite{mcclure-smith:JAMS03} and \cite{bb}.

Let $V$ be a symmetric monoidal closed category. Let $A$ be a small
$V$-category and let $[A,V]$ be the $V$-category of $V$-functors from
$A$ to $V.$ The enriched $\Hom$-functor $\Nat_A(F,G)$ is given by the
end:
$$
\Nat_A(F,G) := \int_{X\in A} V(F(X),G(X)).
$$ 
We also define the tensor product of the $V$-functors
$F:A^{op}\rightarrow V$ and $G:A\rightarrow V$ by the coend
$$
F \otimes_A G := \int^{X\in A} F(X)\otimes G(X).
$$

\begin{definition} 
A {\em $V$-substitude\/}  $(P,A)$ is a small 
$V$-category  $A$ together
with a sequence of $V$-functors:
$$
P_n: \underbrace{A^{op}\otimes\cdots\otimes A^{op}}_{n-times}\otimes  
A \rightarrow V, \  n\ge 0,\ $$ $$P_n(X_1,\ldots,X_n; X) =  
P_{X_1,\ldots,X_n}^ X
$$
equipped with
\begin{itemize}
\item 
a $V$-natural family of substitution operations
$$\mu: P_{X_1,\ldots,X_n}^ X \otimes P_{X_{11},\cdots,X_{1m_1}}^  
{X_1}\otimes\cdots\otimes P_{X_{n1},\ldots,X_{nm_n}}^ {X_n}\rightarrow  
P_{X_{11},\ldots,X_{nm_n}}^ {X}$$
\item 
a $V$-natural family of morphisms (unit of substitude)
$$\eta: A(X,Y)\rightarrow P_1(X;Y) = P_X^Y$$
\item 
for each permutation $\sigma\in S_n$ a $V$-natural family of  
isomorphisms
$$\gamma_{\sigma}:P_{X_1,\ldots,X_n}^X \rightarrow P_{X_{\sigma(1)}, 
\ldots,X_{\sigma(n)}}^X,$$
\end{itemize}
satisfying some associativity, unitality and equivariancy conditions  
\cite{DS2}.
\end{definition}

Notice that $P_1$ is a $V$-monad on $A$ in the bicategory of 
$V$-bimodules ($V$-profunctors or  $V$-distributors). The Kleisli category  
of this monad is called {\it the underlying category of $P.$}

The concept of substitude generalizes operads and
symmetric lax-monoidal categories. Indeed, any coloured
operad $P$ in $V$ with the set of colours $S$ is naturally a
substitude $(P, U(P))$ with $U(P)$ equal the $V$-category with the set
of objects $S$ and the object of morphisms $U(P)(X,Y)= P(X;Y)\in V $.
The substitution operation in the coloured operad $P$ makes the
assignment $P_n(X_1,\ldots,X_n; X) = P_{X_1,\ldots,X_n}^ X $ a functor
$$ 
P_n: \underbrace{U(P)^{op}\otimes\cdots\otimes
U(P)^{op}}_{n-times}\otimes\ U(P) \rightarrow V, \ n\ge 0,\ 
$$ 
and the
sequence of these functors form a substitude. The category $U(P)$ is
the underlying category of this substitude also called the underlying
category of the coloured operad $P.$ In fact, a substitude
is a coloured operad $P$ together with a small $V $-category
$A$ and a $V$-functor $\eta: A \rightarrow U(P)$
\cite[Prop. 6.3]{DS1}.

\begin{definition} \cite{batanin:AM08,DS2}
A {\em symmetric lax-monoidal structure\/} or a {\em multitensor\/} on
a $V$-category $C$ is a sequence of $V$-functors
$$E_n: \underbrace{C\otimes\cdots\otimes C}_{n- times} \rightarrow C$$
equipped with
\begin{itemize}
\item 
a family of  $V$-natural transformations:
$$ \mu:E_n(E_{m_1},\ldots,E_{m_k})\rightarrow E_{m_1+\cdots+m_k};$$
\item 
A $V$-natural transformation (unit)
$${\it Id} \rightarrow E_1;$$
\item  an action of symmetric group
$$\gamma_{\sigma}:E_n(X_1,\ldots,X_n) \rightarrow  
E_n(X_{\sigma^{-1}(1)},\ldots,X_{\sigma^{-1}(n)}),$$
\end{itemize}
satisfying some natural associativity, unitarity and equivariance  conditions.
\end{definition}

\begin{definition} 
\cite{mcclure-smith:JAMS03} A multitensor is called a {\em
functor-operad\/} if its unit is an isomorphism.
\end{definition}

McClure and Smith observed in~\cite{mcclure-smith:JAMS03} that
functor-operads can be used to define operads.  Their observation
works also for multitensors. Let $\delta\in C$ be an object of $C$
then the coendomorphism operad of $\delta$ with respect to a
multitensor $E$ is given by a collection of objects in $V$
$$\Coend^E(\delta)(n) = C(\delta,E_n(\delta,\ldots,\delta)).$$

Substitudes and multitensors are related  by the following convolution 
operation \cite{DS1,DS2}.

\begin{definition}  
Let $(P,A)$ be a substitude.  We define a multitensor $E^P$ on $C=
[A,V]$ as follows:
\begin{equation}
\label{coend} 
E^P_n(\phi_1,\ldots,\phi_n)(X)=  P_{-,\ldots, -}^X\otimes_A \phi_1(-)\otimes_A 
\cdots\otimes_A \phi_n(-).
\end{equation} 
\end{definition}

A special case of this construction is when $A$ is equal to the  
underlying category of $P.$ In this case
the convolution operation produces a functor-operad.

Let $(P,A)$  be a substitude and let $\delta: A\rightarrow V 
$ be a $V$-functor.

\begin{definition} 
By a {\em $\delta$-condensation\/} of the substitude $(P,A)$ we mean the
operad $C^{(P,A)}(\delta) = \Coend^{E^P}\hskip -2.3pt(\delta)$. 
So, as a collection
it is given by
$$ 
C^{(P,A)}(\delta)(n) = \Nat_A(\delta, E_n^P(\delta,\ldots, \delta)).
$$
\end{definition}

The operad $C^{(P,A)}(\delta)$ naturally  
acts on the objects of the form
$$ 
{\it Tot\/}_{\delta}(\phi)= \Nat_A(\delta, \phi)
$$
for an arbitrary $V$-functor  $\phi: A\rightarrow V$ 
($\delta$-totalization of $\phi$)~\cite{mcclure-smith:JAMS03,bb}.

Let $i:B\rightarrow A$ and  $\delta: B\rightarrow V 
$ be two $V$-functors. Let $Lan_i(\delta)$  be a ($V$-enriched) left Kan  
extension of $\delta$ along $i .$ Then
$$
{\it Tot\/}_{Lan_i(\delta)}(\phi) =   
\Nat_A(Lan_i(\delta),  \phi)
= \Nat_B(\delta, i^*(\phi))
 = 
{\it Tot\/}_{\delta} 
(i^*(\phi)),
$$
where $i^*$ is the restriction functor induced by $i .$

There is a similar formula  which expresses the condensation with  
respect to $Lan_i(\delta) .$
Let $(P,A)$ be a substitude and let $i_{*,\ldots,*}(P)$ be a sequence of  
functors
$$ i_{*,\ldots,*}(P)_n: B^{op}\otimes \cdots\otimes B^{op}\otimes A
\rightarrow V,
$$ 
$$ i_{*,\ldots,*}(P)^A_{B_1,\cdots,B_n} = P^A_{i(B_1),\ldots,i(B_n)} \ \  
\ \   \ \  . 
$$
We define a sequence of functors 
$$
{E}_n^{i_{*,\ldots,*}(P)}: [B,V] \otimes \cdots \otimes [B,V]
\rightarrow  [A,V]
$$ 
by the formula similar to formula~(\ref{coend}).
We also define $i^*P$ as the substitude $(i^*P,B)$ obtained from $P$  
by restricting $P_n$ along $i .$

\begin{proposition}
\label{Lancondensation} 
For the functors $\phi_1,\ldots \phi_n \in [B,V]$  
the following $V$-natural isomorphisms hold:
$$ 
E_n^P(Lan_i(\phi_1),\ldots,Lan_i(\phi_n)) = {E}_n^{i_{*,\ldots,*}(P)} 
(\phi_1,\ldots,\phi_n).
$$
In particular,
\begin{eqnarray*}
C^{(P,A)}(Lan_i(\delta))(n)\!\! &=& \!\!
{\it Tot\/}_{\delta}(i^*{E}_n^{i_{*,\ldots,*}(P)} (\delta, 
\ldots,\delta) )
\\
\!\!&=& \! \!{\it Tot\/}_{\delta}( E_n^{i^*(P)}(\delta,\ldots,\delta) ) =  
C^{(i^*(P),B)}(\delta)(n).
\end{eqnarray*}
\end{proposition}

This result allows to see many of the operads in this paper as the
result of $\delta$-condensation of some substitudes. For us $V$ will
be the category of chain complexes ${\it Ch}$.  Our category $A$ will
be the category of nonempty ordinals $\Delta$ (linearized) or the
crossed interval category $(IS)^{op}$~\cite{batanin-markl} (also
linearized).  $B$ can be $\Delta$ or its subcategory of injective
order preserving maps $\Delta_{in} .$ These categories are related by
the canonical inclusions:
$$
\Delta_{in}\stackrel{i}{\longrightarrow} \Delta 
\stackrel{j}{\longrightarrow}(IS)^{op}.
$$ 
Let $\delta: \Delta \rightarrow {\it Ch}$ be the cosimplicial chain
complex of normalized chains on standard simplices. It is classical
that the totalization of a cosimplicial chain complex $X^{\bullet}$
with respect to $\delta$ is the normalized cosimplicial
totalization $\oN(X^{\bullet})$ and the tensor product
$X_{\bullet}\otimes_{\Delta} \delta$ for a simplicial chain complex
$X_{\bullet}$ is the normalized simplicial realization
$\uN(X_{\bullet}).$ Hence, the condensation of the lattice path operad
$\Lat c$ with respect to $\delta$ is precisely the $n$-simplicial
cosimplicial normalization $$|\overline{ \Lat c} | = \oN(\uN(\Lat c
(\Rada \bullet1n;\bullet))) = C^{(\Lat c ,\Delta)}(\delta).$$

Proposition~\ref{Lancondensation} shows that 
the condensation of  the lattice path operad $\Lat c$ with respect to 
$Lan_i(i^*(\delta))$ is the unnormalized $n$-simplicial cosimplicial totalization 
$$|\Lat c | = \oTot(\uTot(\Lat c (\Rada \bullet1n;\bullet))) =  C^{(i^*(\Lat c) ,\Delta_{in})}(i^*(\delta))= C^{(\Lat c ,\Delta)}(Lan_i(i^*(\delta)))   $$

Analogously, for the  operad of natural operations on the
Hochschild cochains we use the condensation with respect to
$Lan_j(\delta)$ for the normalized version and with respect to
$Lan_{ji}(i^*{\delta})$ for the unnormalized version that is
$$\label{bigcon} \Big = C^{(B  ,(IS)^{op})}(Lan_i(i^*(\delta))) .$$
In~\cite{bb} similar
calculations were applied to the cyclic version of the lattice path
operad.

\backmatter

\providecommand{\bysame}{\leavevmode ---\ }
\providecommand{\og}{``}
\providecommand{\fg}{''}
\providecommand{\smfandname}{\&}
\providecommand{\smfedsname}{\'eds.}
\providecommand{\smfedname}{\'ed.}
\providecommand{\smfmastersthesisname}{M\'emoire}
\providecommand{\smfphdthesisname}{Th\`ese}

\end{document}